\documentclass[oneside,english]{amsart}
\usepackage[T1]{fontenc}
\usepackage[latin9]{inputenc}
\usepackage{geometry}
\usepackage{color}
\usepackage{babel}
\usepackage{mathrsfs}
\usepackage{amstext}
\usepackage{amsthm}
\usepackage{amssymb}
\usepackage{stmaryrd}
\usepackage{esint}
\usepackage[unicode=true,pdfusetitle,
 bookmarks=true,bookmarksnumbered=false,bookmarksopen=false,
 breaklinks=false,pdfborder={0 0 1},backref=false,colorlinks=false]
 {hyperref}

\makeatletter
\numberwithin{equation}{section}
\numberwithin{figure}{section}
\theoremstyle{plain}
\newtheorem{thm}{\protect\theoremname}
\theoremstyle{definition}
\newtheorem{definition}[thm]{\protect\definitionname}
\theoremstyle{plain}
\newtheorem{lemma}[thm]{\protect\lemmaname}
\theoremstyle{remark}
\newtheorem{note}[thm]{\protect\notename}
\theoremstyle{remark}
\newtheorem{remark}[thm]{\protect\remarkname}
\theoremstyle{plain}
\newtheorem{proposition}[thm]{\protect\propositionname}
\theoremstyle{plain}
\newtheorem{corr}[thm]{\protect\corollaryname}
\theoremstyle{definition}
\newtheorem{example}[thm]{\protect\examplename}

\usepackage{babel}

\providecommand{\corollaryname}{Corollary}
\providecommand{\definitionname}{Definition}
\providecommand{\examplename}{Example}
\providecommand{\lemmaname}{Lemma}
\providecommand{\notename}{Note}
\providecommand{\propositionname}{Proposition}
\providecommand{\remarkname}{Remark}
\providecommand{\theoremname}{Theorem}

\newcommand{\bq}{\begin{equation}}
\newcommand{\eq}{\end{equation}}

\makeatother

\providecommand{\corollaryname}{Corollary}
\providecommand{\definitionname}{Definition}
\providecommand{\examplename}{Example}
\providecommand{\lemmaname}{Lemma}
\providecommand{\notename}{Note}
\providecommand{\propositionname}{Proposition}
\providecommand{\remarkname}{Remark}
\providecommand{\theoremname}{Theorem}

\begin{document}
\title{Linear Boundary Port-Hamiltonian Systems with Implicitly Defined Energy}
\author{B. Maschke $^{(1)}$ and A.J. van der Schaft $^{(2)}$ }
\address{$^{(1)}$ Université Claude Bernard Lyon 1, CNRS, LAGEPP UMR 5007,
France\\
$^{(2)}$ Bernoulli institute for Mathematics, CS and AI, University
of Groningen, the Netherlands }
\email{bernhard.maschke@univ-lyon1.fr; a.j.van.der.schaft@rug.nl}

\begin{abstract}
In this paper we extend the previously introduced class of boundary
port-Hamiltonian systems to boundary control systems where the variational
derivative of the Hamiltonian functional is replaced by a pair of
reciprocal differential operators. In physical systems modelling,
these differential operators naturally represent the constitutive
relations associated with the implicitly defined energy of the system
and obey Maxwell's reciprocity conditions. On top of the boundary
variables associated with the Stokes-Dirac structure, this leads to
additional boundary port variables and to the new notion of a Stokes-Lagrange
subspace. This extended class of boundary port-Hamiltonian systems
is illustrated by a number of examples in the modelling of elastic
rods with local and non-local elasticity relations. Finally it is shown
how a Hamiltonian functional on an extended state space can be associated
with the Stokes-Lagrange subspace, and how this leads to an energy
balance equation involving the boundary variables of the Stokes-Dirac
structure as well as of the Stokes-Lagrange subspace. 
\end{abstract}

\maketitle

\section{Introduction}
Port-Hamiltonian systems theory concerns \emph{open} physical
systems. It aims at developing a theoretical framework for modelling,
numerical simulation and control, based on the axioms of macroscopic
physics (including mechanics, electro-magnetism, irreversible thermodynamics),
network theory, as well as the geometric methods of mathematical physics
and control theory \cite{Geoplex09, Jeltsema_FTSC_14}. Port-Hamiltonian systems may
be regarded as the extension of classical Hamiltonian systems \cite[chap.3]{Arnold89}\cite[chap.3]{Libermann_marle87}
to open systems. This is done by including pairs of interface variables,
called \emph{port variables}, and extending the geometric structure
of Hamiltonian systems, namely the Poisson bracket, to a \emph{Dirac
structure} defined on spaces also containing the port variables \cite{schaftAEU95}\cite[chap.2]{Geoplex09}\cite{Jeltsema_FTSC_14}.
For infinite-dimensional port-Hamiltonian systems, a Dirac structure,
called \emph{Stokes-Dirac structure} \cite{schaftGeomPhys02}, is derived from the formally
skew-adjoint Hamiltonian differential operator \cite[Chap.7]{Olver93} of the system. This
allows to define pairs of \emph{boundary} port variables in which
all physically consistent boundary conditions, boundary control and
observation variables may be expressed \cite{schaftGeomPhys02}\cite[chap.4]{Geoplex09}
\cite{Rashad_IMA18_20yearsBPHS}. The development of the resulting class of boundary port-Hamiltonian systems has led to numerous
applications in modelling, simulation and control of conservative
systems (see e.g.\cite{MacchelliSIAM04,Hamroun10,Brugnoli_AMM2019_I}),
as well as of their extension to dissipative systems; retaining the
formulation based on Stokes-Dirac structures (see e.g.\cite{TrangVU15_2_MCMDS,Vincent_IFAC_WC20_PhaseFields,Baaiu_MCMDS_09}).
 In particular linear boundary port-Hamiltonian systems defined on  one-dimensional
spatial domains have led to a successful collection of work using
semi-group operator theory \cite{LeGorrecSIAM05,JacobZwart12,Macchelli_2021OverviewControlBPHS}.

However, this approach to boundary port-Hamiltonian systems relies on the assumption that the Hamiltonian
system is defined by a Hamiltonian differential operator together with a generating Hamiltonian functional that only depends
on the state variables and \emph{not} on their spatial derivatives.
First, this does not incorporate the commonly used symplectic formulations
of elastic mechanical systems, since in this case the Hamiltonian
functional depends on the \emph{strain}, which is the spatial derivative
of the position. In fact, such systems were expressed as boundary
port-Hamiltonian systems by \emph{extending} the state space to the
jet bundle on the positions \cite{LeGorrecSIAM05,MASCHKE_CPDE_2013,Brugnoli_AMM2019_I,Brugnoli_AMM2019_II}.
Second, this formulation also excludes non-local elasticity relations
as appearing e.g. in composite materials \cite{Zwart_MCMDS19_nonlocal}. 

In this paper, we shall extend the definition of boundary port-Hamiltonian
systems in order to address these two issues. This will be done by
starting from the definition of port-Hamiltonian
systems defined on Lagrangian subspaces \cite{SCL_18,schaft_ArXiv22_ImplPHSMonotoneDiss}
or Lagrangian submanifolds \cite{Schaft2020_DiracLagrangeNonlinear}
as recently introduced in the finite-dimensional case. In this definition, the energy is
no longer defined by a function on the state space, but instead by
reciprocal constitutive relations between the state and co-state variables.
Geometrically these relations define Lagrangian subspaces or submanifolds
of the co-tangent bundle of the state space \cite[p.414]{Abraham_marsden87}.
In this paper we shall extend this idea to the infinite-dimensional case.
We will restrict to \emph{linear} {boundary port-Hamiltonian systems defined
on one-dimensional spatial domains, thereby extending the setting
of \cite{LeGorrecSIAM05}. Some preliminary results (in restricted
contexts) have been presented already in \cite{schaft_LHMNLC21_DiffDiracOp,Maschke_IFAC_WC20_IPHS}.

The paper is structured as follows. In Section 2, we recall the definition
of boundary port-Hamiltonian systems defined with respect to Stokes-Dirac
structures, as given in \cite{schaftGeomPhys02,LeGorrecSIAM05}; however
using the algebraic framework of two-variable polynomial matrices
(see already \cite{schaft_LHMNLC21_DiffDiracOp}). Furthermore, we
shall motivate the developments of the paper by considering various
models of an elastic rod with local and non-local elasticity relations.
In Section 3, we develop the definition of Lagrangian subspaces associated
with reciprocal differential operators. We show that, for open systems,
one has to extend the Lagrangian subspaces to include boundary variables,
leading to the definition of Stokes-Lagrange subspaces. In Section
4, we define linear boundary port-Hamiltonian systems defined with
respect to these Stokes-Lagrange subspaces. Furthermore, we derive
a Hamiltonian functional, and show that this functional satisfies
a balance equation. This is illustrated with the examples
of the elastic rod in the symplectic formulation and the elastic rod
with non-local elasticity relations. 

\section{\label{sec:Reminder2PolynomialCalculus} Reminder on functionals,
differential operators and polynomial matrix calculus}

In this section we briefly recall from \cite{Willems_SIAM98,Schaft_SIAM11}
the relevant definitions and results on two-variable polynomial matrices
and the matrix differential operators associated with them. First we recall that a matrix differential operator in a variable $z \in \mathbb{R}$ is given by an expression
\begin{equation}
A(\frac{\partial}{\partial z}) = \sum_{k=0}^{N} A_{k}\frac{\partial^{k}}{\partial z^{k}}
\end{equation}
for $p \times q$ matrices $A_k$. If $A_N \neq 0$ then the integer $N$ is called the order of the matrix differential operator. Note that $A(\frac{\partial}{\partial z})$ maps any $C^{\infty}\left(\left[a,b\right],\mathbb{R}^{p}\right)$ function to an $C^{\infty}\left(\left[a,b\right],\mathbb{R}^{q}\right)$ function. 
Associated with $A(\frac{\partial}{\partial z})$ is the $p \times q$ polynomial matrix (in one variable $s$) 
\begin{equation}
A(s) = \sum_{k=0}^{N} A_{k}s^k
\end{equation}
The \emph{formal adjoint} of the $p \times q$ matrix differential operator $A(\frac{\partial}{\partial z}) = \sum_{k=0}^{N} A_{k}\frac{\partial^{k}}{\partial z^{k}}$ is the $q \times p$ matrix differential operator
\begin{equation}
A^*(\frac{\partial}{\partial z}) := \sum_{k=0}^{N} (-1)^kA^\top_{k}\frac{\partial^{k}}{\partial z^{k}},
\end{equation}
satisfying
\begin{equation}
\int_{a}^{b}e_{1}^{\top}(z) A^* e_{2}(z)dz=\int_{a}^{b}e_{2}^{\top}(z) Ae_{1}(z)dz
\end{equation}
for any pair $e_1\in C^{\infty}\left(\left[a,b\right],\mathbb{R}^{q}\right)$ and $e_2\in C^{\infty}\left(\left[a,b\right],\mathbb{R}^{p}\right)$
 with support strictly contained in the interval $[a,b]$. The $q \times p$ polynomial matrix associated with $A^*(\frac{\partial}{\partial z})$ is given by $A^\top(-s) = \sum_{k=0}^{N} A^\top_{k} (-s)^k$.\\
Finally, a $p \times p$ matrix differential operator is called \emph{formally self-adjoint} if $A(\frac{\partial}{\partial z})=  A^*(\frac{\partial}{\partial z})$, or equivalently $A(s)=A^\top(-s)$, and \emph{formally skew-adjoint} if $A(\frac{\partial}{\partial z})=  - A^*(\frac{\partial}{\partial z})$, or equivalently $A(s)=-A^\top(-s)$.

Next we consider \emph{two-variable} polynomial matrices and their corresponding bilinear differential operators.
\begin{definition}
A $p\times q$ \emph{two-variable polynomial matrix} $\Phi(\zeta,\eta)$ in the two indeterminates $\zeta$
and $\eta$ is an expression of the form 
\begin{equation}
\Phi(\zeta,\eta):=\sum_{k,l=0}^{M}\Phi_{k,l}\zeta^{k}\eta^{l}\label{eq:Def2PolynomialMatrices}
\end{equation}
with $M$ a nonnegative integer and $\Phi_{k,l}\in\mathbb{R}^{p\times q}$ $p\times q$ matrices.  
It is equivalently defined by its \textit{coefficient matrix} $\widetilde{\Phi}$ which
is the $Mp \times Mq$ matrix whose $(k,l)$-th block
is the matrix $\Phi_{k,l}$, $k,l=0,\ldots,M,$. The smallest $M$ such that $\Phi_{k,l}=0$ for all $k,l>M$ is called the \emph{degree} of $\Phi(\zeta,\eta)$.
\end{definition}

First, two-variable polynomial matrices can be used to define bilinear differential operators.
\begin{definition}
The \textit{bilinear differential operator}\footnote{Called a bilinear differential form in \cite{Willems_SIAM98}.}
$D_{\Phi}$ associated to the two-variable polynomial matrix $\Phi(\zeta,\eta)$
in (\ref{eq:Def2PolynomialMatrices}) is the operator 
\begin{equation}
D_{\Phi}(v,w)(z)=\sum_{k,l=0}^{M}\left[\frac{d^{k}}{dz^{k}}v(z)\right]^{\top}\Phi_{k,l}\frac{d^{l}}{dz^{l}}w(z),\label{eq:BilinearDiffOperator-1}
\end{equation}
acting on vector-valued $M$-differentiable functions $v:\mathbb{R}\to\mathbb{R}^{p},w:\mathbb{R}\to\mathbb{R}^{q}$.
This operator defines, in turn, a \emph{bilinear form} on the functions
with support being a bounded one-dimensional spatial domain $Z=\left[a,\,b\right]\ni z$,
where $a,\,b\in\mathbb{R}$ with $a<b$, given as 
\begin{equation}
\mathbf{D}_{\Phi}(v,w)=\int_{a}^{b}D_{\Phi}(v,w)(z)\,dz.\label{eq:BilinearDiffForm}
\end{equation}
\end{definition}

Second, we recall how two-variable polynomial matrices can be used
to express \emph{'integration by parts'}, or equivalently the product
rule of differentiation, in a compact way \cite[Section 2.2]{Schaft_SIAM11}.
These results will be employed in the sequel to define the boundary maps
for port-Hamiltonian systems. 
\begin{lemma}
\label{lem:Derivative-of-QuadraticDiffForm} The \emph{derivative
of the bilinear operator} $D_{\Psi}$ associated with the two-variable
polynomial matrix $\Psi(\zeta,\eta)$ is the bilinear differential
operator 
\begin{equation}
D_{\Phi}(v,w)(z) =\frac{d}{dz}\left(D_{\Psi}(v,w)\right)(z)\,,\label{eq:DerivativeDiffOperator}
\end{equation}
associated with the $p\times q$ two-variable polynomial matrix $\Phi(\zeta,\eta)$ given by 
\begin{equation}
\Phi(\zeta,\eta)=(\zeta+\eta)\Psi(\zeta,\eta)\label{eq:DerivativeOperatorCoeffcientMatrix}
\end{equation}
Conversely, a bilinear
operator $D_{\Phi}(v,w)(z)$ is the derivative of another bilinear differential operator if and only if the two-variable polynomial matrix $\Phi(\zeta,\eta)$ is divisible by $(\zeta+\eta)$ (i.e., \eqref{eq:DerivativeOperatorCoeffcientMatrix} holds for some $\Psi(\zeta,\eta)$).
\end{lemma}

 In integral form, one derives from \eqref{eq:DerivativeDiffOperator} the expression 
\begin{equation}
\mathbf{D}_{\Phi}(v,w)=D_{\Psi}(v,w)(b)-D_{\Psi}(v,w)(a)=:\left[D_{\Psi}(v,w)(z)\right]_{a}^{b}\label{eq:IntegrationbyParts}
\end{equation}
Thus the operator $D_{\Psi}$ defines the boundary terms ('remainders') obtained when 'integrating
by parts' the bilinear form $\mathbf{D}_{\Phi}$ in (\ref{eq:BilinearDiffForm}) with $D_{\Phi}$ as in \eqref{eq:BilinearDiffOperator-1}.

Third, we shall use the following basic observation concerning factorization of two-variable polynomial matrices into one-variable polynomial matrices}\cite[p. 1709]{Willems_SIAM98}, \cite[Proposition 2.3]{Schaft_SIAM11}. Consider the $p\times q$ two-variable polynomial matrix $\Phi(\zeta,\eta)$ with $Mp \times Mq$ coefficient matrix $\widetilde{\Phi}$. Consider any factorization $\widetilde{\Phi} = \widetilde{X}^\top \widetilde{Y}$ of $\widetilde{\Phi}$, where $\widetilde{X}$ is a $k \times Mp$ matrix and $\widetilde{Y}$ is a $k \times Mq$ matrix for some nonnegative integer $k$. Define the $k \times p$, respectively $k\times q$, one-variable polynomial matrices
\begin{equation}
X(s) := X_0 + X_1s + \cdots + X_Ms^M, \quad Y(s) := Y_0 + Y_1s + \cdots + Y_Ms^M,
\end{equation}
where 
\begin{equation}
\widetilde{X}= \begin{bmatrix} X_0 & X_1 & \cdots & X_M\end{bmatrix}, \quad  \widetilde{Y}= \begin{bmatrix} Y_0 & Y_1 & \cdots & Y_M\end{bmatrix}.
\end{equation}
Then it is immediately verified that $\widetilde{\Phi}(\zeta,\eta)=X^\top (\zeta)Y(\eta)$.

Factorizations which correspond to the minimal value $k=\textrm{rank}(\widetilde{\Phi})$
are called \textit{minimal}. They are unique up to premultiplication
by a constant nonsingular matrix.

Fourth, we recall the definition of the differential of a symmetric
bilinear functional form (\ref{eq:BilinearDiffForm}) \cite[section 3]{Willems_SIAM98}.
A $p\times p$ two-variable polynomial matrix $H(\zeta,\eta)$ is
called \textit{symmetric} if $H(\zeta,\eta)=H^{\top}(\eta,\zeta)$,
or equivalently its coefficient matrix $\widetilde{H}$ is symmetric.
A symmetric two-variable polynomial matrix $H(\zeta,\eta)$ defines
the \emph{quadratic differential operator}\footnote{Called a quadratic differential form in \cite{Willems_SIAM98}.} $D_{H}(v,v)(z)$ acting
on vector-valued $M$-differentiable functions $v:\mathbb{R}\to\mathbb{R}^{p}$. Note that for any such function $v$ the expression
\begin{equation}
\mathfrak{H}\left(v\right):=\frac{1}{2}\mathbf{D}_{H}(v,v)\label{eq:QuadraticForm}
\end{equation}
is a \emph{quadratic functional}. Using the bilinearity and symmetry
of $D_{H}$, the following proposition may
be proven easily. 
\begin{lemma} \label{lem:DiffQuadFunctPolynom}
Consider the \emph{quadratic functional} $\mathfrak{H}\left(v\right)$
in (\ref{eq:QuadraticForm}), defined by a symmetric two-variable
polynomial matrix $H(\zeta,\eta)$. For any $\varepsilon\in\mathbb{R}$
and any $M$-differentiable function $\delta v:\mathbb{R}\to\mathbb{R}^{p}$
\begin{equation}
\mathfrak{H}(v+\varepsilon\delta v)=\mathfrak{H}(v)+\varepsilon\int_{a}^{b}D_{H}\left(v,\delta v\right)dz+O\left(\varepsilon^{2}\right).\label{eq:DifferentialQuadraticForm}
\end{equation}
Hence $D_{H}\left(v,.\right):\delta v\mapsto D_{H}\left(v,\delta v\right)$
is the \emph{differential} of the quadratic functional $\mathfrak{H}\left(v\right)$. 
\end{lemma}

The differential of the quadratic functional $\mathfrak{H}\left(v\right)$
is related to its variational derivative \footnote{The \emph{variational derivative} $\frac{\delta\mathfrak{H}}{\delta v}\left(v\right)$ of $\mathfrak{H}\left(v\right)$ is defined as the differential operator acting on $v$ such that 
\begin{equation}
\mathfrak{H}(v+\varepsilon\delta v)=\mathfrak{H}(v)+\varepsilon\int_{a}^{b} \delta v^\top \,\frac{\delta\mathfrak{H}}{\delta v}\left(v\right)\,dz+O\left(\varepsilon^{2}\right)\label{eq:VariationalDerivative}
\end{equation}
} \cite[page 249]{Olver93} in the following way. Repeated integration
by parts for functions \emph{with compact support strictly included in $[a, b]$}, yields 
\begin{equation}
\int_{a}^{b}D_{H}\left(v,\delta v\right)dz=\int_{a}^{b} \delta v^\top \,Q(\frac{d}{dz})(v)\,dz\label{eq:Differential_VariationalDerivative_Relation}
\end{equation}
where 
\begin{equation}
Q(s):=H(-s,s)\label{eq:VariationalDer_Polynomial}
\end{equation}
Thus the variational derivative of $\mathfrak{H}\left(v\right)$ is
given by $\frac{\delta\mathfrak{H}}{\delta v}\left(v\right)=Q(\frac{d}{dz})\left(v\right)$.
Since $H(\zeta,\eta)=H^\top (\eta,\zeta)$ one has $Q(s)=Q^{\top}(-s)$, and thus the differential operator $Q(\frac{d}{dz})$ is formally self-adjoint.
\begin{note}
For \emph{general functions, not vanishing at the boundary of the spatial domain,} integration by parts yields
\begin{equation}
\int_{a}^{b}D_{H}\left(v,\delta v\right)dz=\int_{a}^{b} \delta v^\top \,Q(\frac{d}{dz})\left(v\right)\,dz+\left[\sum_{k,l=0}^{M}\left[\frac{d^{k}}{dz^{k}}v\right]^{\top}H_{k,l}^{b}\left[\frac{d^{l}}{dz^{l}}\delta v\right]\right]_{a}^{b}\label{Differential_VariationalDerivative_BoundaryTerm}
\end{equation}
and the two-variable polynomial equation expressing integration by
parts \textit{with} boundary terms ('remainders') becomes 
\[
H(\zeta,\eta)-Q(\eta)=(\zeta+\eta)H^{b}(\zeta,\eta)
\]
for some (usually \textit{not} symmetric) two-variable polynomial
matrix $H^{b}(\zeta,\eta)$. 
\end{note}

\section{\label{sec:LinearBPHS_Motivation} Linear boundary port-Hamiltonian
systems and motivating examples}

This section is composed of two parts. In the first part, we shall
briefly recall the definition of linear boundary port-Hamiltonian
systems defined on one-dimensional spatial domains. This class of
systems was defined in \cite{LeGorrecSIAM05}, extending \cite{schaftGeomPhys02},
and has been the topic of numerous work on their analysis and control \cite{JacobZwart12,Macchelli_2021OverviewControlBPHS}. However,
we shall present them in a slightly different way, using the algebraic framework
of two-variable polynomial matrices recalled in Section \ref{sec:Reminder2PolynomialCalculus}.
In the second part of the section, we shall motivate the main idea
of this paper, i.e., the extension from Hamiltonian functionals to
Lagrangian subspaces, on the basis of different formulations of the
model of an elastic rod. We shall recall why the symplectic formulation
of the elastic rod \cite[chap.4]{Geoplex09} does not correspond to
the definition of linear boundary port-Hamiltonian systems as defined
in \cite{schaftGeomPhys02,LeGorrecSIAM05}. Then we illustrate how
the elastic rod may be expressed in terms of linear boundary port-Hamiltonian
systems, using a prolongation of the state space. Finally for elastic
rods with a \emph{non-local} elasticity relation, we recall how one
obtains a descriptor representation \cite{Zwart_MCMDS19_nonlocal}
which also is not encompassed by the standard definition of linear
boundary port-Hamiltonian systems.

\subsection{\label{subsec:Boundary-Port-HamiltonianSyst} Linear boundary port-Hamiltonian
systems}

In this paragraph, we recall from \cite{schaftGeomPhys02, LeGorrecSIAM05} the definition of \emph{linear
boundary port-Hamiltonian systems} on a bounded one-dimensional spatial domain $Z=\left[a,\,b\right]\ni z$,
$a,\,b\in\mathbb{R}$ and $a<b$ ; however using
the framework of two-variable polynomial matrices.

\subsubsection{Hamiltonian systems}

We consider \emph{Hamiltonian systems} \cite[chap.7]{Olver93} defined
with respect to a Hamiltonian operator $\mathcal{J}\left(\frac{\partial}{\partial z}\right)$, which is a
formally skew-symmetric $N$th-order matrix differential operator
\begin{equation}
\mathcal{J}(\frac{\partial}{\partial z})=\sum_{k=0}^{N}J_{k}\frac{\partial^{k}}{\partial z^{k}}\label{eq:N_OrderLinearHamOp}
\end{equation}
where $J_{k}\in\mathbb{R}^{n\times n}$ is symmetric if $k$ is odd
and skew-symmetric if $k$ is
even\footnote{The operator $\mathcal{J}$ is called a \emph{Hamiltonian operator} since
it is formally skew-adjoint and, being defined by constant coefficient
matrices, also satisfies the Jacobi identity \cite[Corollary 7.5, p. 429]{Olver93}.}, and generated by the Hamiltonian functional 
\begin{equation}
\mathfrak{H}\left[x\right]=\int_{a}^{b}\frac{1}{2}x\left(z\right){}^{\top}Q_{0}\,x\left(z\right)\;dz\label{eq:EnergyFunctionQuadratic},
\end{equation}
where $Q_{0}$ is a constant symmetric matrix and $x\in L^{2}\left(\left[a,b\right],\mathbb{R}^{n}\right)$
is the vector of state variables. Note that, in view of the physical
interpretation of these models, we will also call the state variables
the \emph{energy variables} and the Hamiltonian functional the \emph{energy}.
The resulting Hamiltonian system is then defined by the partial differential
equation (PDE) 
\begin{equation}
\frac{\partial x}{\partial t}=\mathcal{J}(\frac{\partial}{\partial z})\,Q_{0}\,x\label{eq:LinearHamiltonianSys}
\end{equation}
Indeed, this PDE may be written as 
\begin{equation}
\frac{\partial x}{\partial t}=\mathcal{J}(\frac{\partial}{\partial z})\,\frac{\delta H}{\delta x}\label{eq:HamiltonianSystem}
\end{equation}
where the \emph{variational derivative (\ref{eq:VariationalDerivative})}
(also called \emph{functional derivative}) of the Hamiltonian functional
(\ref{eq:EnergyFunctionQuadratic}) is the \emph{co-energy variable}
\begin{equation}
e=\frac{\delta H}{\delta x}=Q_{0}\,x\label{eq:CoenergyExplicitLinear}
\end{equation}

In order to define \emph{port variables} describing the interaction
of the system with its environment through the boundary of the spatial
domain, these Hamiltonian systems have been extended to \emph{boundary
port-Hamiltonian systems} defined with respect to a Dirac structure
uniquely associated with the Hamiltonian operator (\ref{eq:N_OrderLinearHamOp})
and called \emph{Stokes-Dirac structure} \cite{schaftGeomPhys02,LeGorrecSIAM05}.

\subsubsection{Construction of the boundary port variables associated with the Hamiltonian
operator $\mathcal{J}$}

Let us now recall the definition of \emph{Stokes-Dirac structures}
for a Hamiltonian operator $\mathcal{J}$ as in (\ref{eq:N_OrderLinearHamOp}), slightly departing
from the description in \cite{schaftGeomPhys02,LeGorrecSIAM05} by using the approach
presented in \cite{schaft_LHMNLC21_DiffDiracOp}.

Define the following two-variable polynomial matrix $\Phi(\zeta,\eta)$ based on $\mathcal{J}(s)$:
\begin{equation}
\Phi(\zeta, \eta):= \mathcal{J}(\eta) + \mathcal{J}^\top (\zeta)
\label{eq:2PolynomMatrix_J}
\end{equation}
Clearly, $\Phi(\zeta,\eta)$ is symmetric, while furthermore, since $\mathcal{J}(s)=-\mathcal{J}^\top (-s)$, it vanishes for $\zeta + \eta=0$. Hence
\begin{equation}
\Phi(\zeta, \eta) = (\zeta + \eta) \Psi(\zeta, \eta)
\end{equation}
for some symmetric two-variable polynomial matrix $\Psi(\zeta, \eta)$.

Next, consider a minimal factorization of $\Psi\left(\zeta,\eta\right)$
\begin{equation}
\Psi\left(\zeta,\eta\right)=T\left(\zeta\right)^{\top}\Sigma \; T\left(\eta\right)\label{eq:Boundary2Polynomial_J_Factorization}
\end{equation}
where $\Sigma$ is a full-rank $\delta\times\delta$ symmetric matrix and $T\left(s\right)$ is a $\delta\times n$ polynomial
matrix. Since $\Sigma$ is symmetric, one may always choose $T\left(s\right)$
in such a way that $\Sigma$ becomes the signature matrix 
\begin{equation}
\Sigma=\left(\begin{array}{cc}
I_{\alpha} & 0_{\alpha\times\beta}\\
0_{\beta\times\alpha} & -I_{\beta},
\end{array}\right)\label{eq:SignatureMatrix}
\end{equation}
where $I_{k}$ denotes the identity matrix of dimension $k$, and
$\alpha$, $\beta$ are two integers such that $\alpha+\beta=\delta$.

The boundary term as in (\ref{eq:IntegrationbyParts}) is then written
as 
\begin{equation}
\left[D_{\Psi}\left(e_{1},e_{2}\right)\right]_{a}^{b}=\mathrm{tr}\left(T\left(\frac{\partial}{\partial z}\right)e_{1}\right)^{\top}\left(\begin{array}{cc}
\Sigma & 0_{\delta\times\delta}\\
0_{\delta\times\delta} & -\Sigma
\end{array}\right)\mathrm{tr}\left(T\left(\frac{\partial}{\partial z}\right)e_{2}\right)\label{eq:BoundaryTerm_J}
\end{equation}
where $\mathrm{tr}$ denotes the trace operator: 

\begin{equation} \label{eq:TraceOperatorDef}
\mathrm{tr}\left(\varphi\right)=\left(\begin{array}{c}
\varphi\left(a\right)\\
\varphi\left(b\right)
\end{array}\right)
\end{equation}

In the sequel, we shall use the polynomial matrix $T\left(s\right)$
to define the boundary port variables $\left(f_{\partial},\,e_{\partial}\right)\in\mathbb{R}^{\delta}\times\mathbb{R}^{\delta}$
as follows 
\begin{equation}
\left(\begin{array}{c}
f_{\partial}\\
e_{\partial}
\end{array}\right)=\frac{1}{\sqrt{2}}\left(\begin{array}{cc}
\Sigma & \Sigma\\
\Sigma & -\Sigma
\end{array}\right)\mathrm{tr}\left(T\left(\frac{\partial}{\partial z}\right)e\right)\label{eq:BoundaryPortVariablesSDS}
\end{equation}

A simple calculation, using $\Sigma^{2}=I_{\delta}$, shows
that the boundary term has the following expression in terms of the
boundary port variables (\ref{eq:BoundaryPortVariablesSDS}) associated
with $e_{1}$ and $e_{2}$ 
\begin{equation}
\left[D_{\Psi}\left(e_{1},e_{2}\right)\right]_{a}^{b}=\left(\begin{array}{c}
f_{\partial_{2}}\\
e_{\partial_{2}}
\end{array}\right)^{\top}\left(\begin{array}{cc}
0_{\delta\times\delta} & I_{\delta}\\
I_{\delta} & 0_{\delta\times\delta}
\end{array}\right)\left(\begin{array}{c}
f_{\partial_{1}}\\
e_{\partial_{1}}
\end{array}\right)=f_{\partial_{2}}^{\top}e_{\partial_{1}}+e_{\partial_{2}}^{\top}f_{\partial_{1}}\label{eq:BoundaryTermPortBoundaryVariables}
\end{equation}

\begin{remark}
This definition is equivalent to the definition given in \cite[p.159]{VillegasPhD07,LeGorrecSIAM05}.
Of course the definition of port variables satisfying (\ref{eq:BoundaryTermPortBoundaryVariables})
is not unique. Indeed, a more special definition of the boundary port
variables has been given in \cite{schaft_LHMNLC21_DiffDiracOp}, under
the condition that the matrix $\Sigma$ in (\ref{eq:Boundary2Polynomial_J_Factorization})
has equal number of positive and negative eigenvalues and hence is
equivalent to the signature matrix (\ref{eq:SignatureMatrix}) with
$\alpha=\beta$ (and $\delta=2\alpha$). In fact \cite{trentelman}, this is the case if and only if the coefficient matrix of $\Psi$ has as many positive as negative eigenvalues. In this situation, one may define the
boundary port variables as 
\begin{equation}
\left(\begin{array}{c}
f_{\partial}\\
e_{\partial}
\end{array}\right)=\frac{1}{\sqrt{2}}\left(\begin{array}{cc}
\begin{array}{cc}
I_{\alpha} & I_{\alpha}\\
0_{\alpha} & 0_{\alpha}
\end{array} & \begin{array}{cc}
0_{\alpha\times\alpha} & 0_{\alpha\times\alpha}\\
I_{\alpha} & I_{\alpha}
\end{array}\\
\begin{array}{cc}
I_{\alpha} & -I_{\alpha}\\
0_{\alpha\times\alpha} & 0_{\alpha\times\alpha}
\end{array} & \begin{array}{cc}
0_{\alpha\times\alpha} & 0_{\alpha\times\alpha}\\
I_{\alpha} & -I_{\alpha}
\end{array}
\end{array}\right)\mathrm{tr}\left(T\left(\frac{\partial}{\partial z}\right)e\right)\label{eq:BoundaryPortVariablesSDS_even}
\end{equation}
obviously satisfying (\ref{eq:BoundaryTermPortBoundaryVariables}).
Denoting the first $\alpha$ components of the flow variable $f_{\partial}$
by $f_{\partial}^{a}$ (since they depend on $T\left(\frac{\partial}{\partial z}\right)e\left(a\right)$)
and the last $\alpha$ components by $f_{\partial}^{b}$ (since they
depend on $T\left(\frac{\partial}{\partial z}\right)e\left(b\right)$)
and using a similar notation for the effort variables $e$, the product
of the effort and flow variable may be split into two parts associated
with the boundary points $a$ and $b$ only: $f_{\partial_{2}}^{\top}e_{\partial_{1}}=f_{\partial_{2}}^{a\top}e_{\partial_{1}}^{a}+f_{\partial_{2}}^{b\top}e_{\partial_{1}}^{b}$.
This splitting corresponds to a splitting of the vector $T\left(\frac{\partial}{\partial z}\right)e$ into flow and effort
variables,
in contrast to (\ref{eq:BoundaryPortVariablesSDS}) where the boundary
port variables are defined as a \emph{linear combination of the values}
of $T\left(\frac{\partial}{\partial z}\right)e$ at \emph{both} boundary
points $a$ and $b$ of the spatial domain. 

On the other hand, let us consider
the simplest example of an \emph{odd-dimensional} matrix $\Sigma$, namely
the case of a scalar shift operator $\mathcal{J}=\frac{\partial}{\partial z}$.
In this case, $\left[D_{\Psi}\left(e_{1},e_{2}\right)\right]_{a}^{b}=e_{1}\left(b\right)e_{2}\left(b\right)-e_{1}\left(a\right)e_{2}\left(a\right)$. Hence $\Sigma=1$ and $T=1$, while the effort variable $e\left(z\right)$
is scalar and cannot be decomposed into a pair of flow and effort
variables. Hence the general Definition (\ref{eq:BoundaryPortVariablesSDS}) has
to be used, leading to $\left(\begin{array}{c}
f_{\partial}\\
e_{\partial}
\end{array}\right)=\frac{1}{\sqrt{2}}\left(\begin{array}{c}
e\left(b\right)+e\left(a\right)\\
e\left(b\right)-e\left(a\right)
\end{array}\right)$ . 
\end{remark}

\subsubsection{Stokes-Dirac structures}

Let us now recall the definition of the Stokes-Dirac structure associated
with the Hamiltonian operator $\mathcal{J}$ in (\ref{eq:N_OrderLinearHamOp}). 

\begin{proposition}
\cite[p.159]{VillegasPhD07}\label{prop:StokesDiracStrucNorderLinearOrderOp}Consider
the flow space $\mathscr{F}=L^{2}\left(\left[a,\,b\right],\,\mathbb{R}^{n}\right)\times\mathbb{R}^{p}\ni\left(f^{i},\,f_{\partial}^{i}\right)$
and the effort space being is its dual $\mathscr{E}=\mathscr{F}^{*}\sim L^{2}\left(\left[a,\,b\right],\,\mathbb{R}^{n}\right)\times\mathbb{R}^{p}\ni\left(e^{i},\,e_{\partial}^{i}\right)$. The bond space $\mathscr{B}=\mathscr{F}\times\mathscr{E}$ is
equipped with the symmetric pairing 
\begin{equation}
\left\langle \left(\begin{array}{c}
f\\
f_{\partial}
\end{array}\right),\left(\begin{array}{c}
e\\
e_{\partial}
\end{array}\right) \right\rangle_\mathscr{B}
=
\left\langle f,e\right\rangle _{L^{2}}-\left\langle f_{\partial},e_{\partial}\right\rangle _{\mathbb{R}^{n}}
\end{equation}
where $\left\langle ,\,\right\rangle _{L^{2}}$ denotes the $L^2$ inner product and $\left\langle ,\,\right\rangle _{\mathbb{R}^{n}}$
the Euclidean inner product.\\
Consider a Hamiltonian operator $\mathcal{J}$ defined in (\ref{eq:N_OrderLinearHamOp})
and the associated polynomial matrix $T\left(\zeta\right)$ in (\ref{eq:Boundary2Polynomial_J_Factorization}) 
and $\Sigma$ being (\ref{eq:SignatureMatrix}). The subspace 
\begin{eqnarray}
\mathscr{D}_{\mathcal{J}} & = & \left\{ \left(\left(\begin{array}{c}
f\\
f_{\partial}
\end{array}\right),\left(\begin{array}{c}
e\\
e_{\partial}
\end{array}\right)\right)\in\mathscr{B}\mid e\in H^{N}\left(\left[a,\,b\right],\,\mathbb{R}^{n}\right),\right.\nonumber \\
 &  & \left.f=\mathcal{J}e,\;\left(\begin{array}{c}
f_{\partial}\\
e_{\partial}
\end{array}\right)=\frac{1}{\sqrt{2}}\left(\begin{array}{cc}
\Sigma & \Sigma\\
\Sigma & -\Sigma
\end{array}\right)\mathrm{tr}\left(T\left(\frac{\partial}{\partial z}\right)e\right)\right\} \label{eq:StokesDiracConstEq}
\end{eqnarray}
is a Dirac structure, called \emph{Stokes-Dirac structure}, with respect to the symmetric pairing (\ref{eq:SymmPairingBoundary})
\begin{eqnarray}
\ll \left(
 \left(\begin{array}{c}
f_{1}\\
f_{\partial_{1}}
\end{array}\right),\left(\begin{array}{c}
e_{1}\\
e_{\partial_{1}}
\end{array}\right)
\right)
,\,
\left(
\left(\begin{array}{c}
f_{2}\\
f_{\partial_{2}}
\end{array}\right),\left(\begin{array}{c}
e_{2}\\
e_{\partial_{2}}
\end{array}\right)
\right)
\gg_{+} & =  & \left\langle \left(\begin{array}{c}
f_{1}\\
f_{\partial_{1}}
\end{array}\right),\left(\begin{array}{c}
e_{2}\\
e_{\partial_{2}}
\end{array}\right) \right\rangle_\mathscr{B}\nonumber \\
 &  & + \left\langle \left(\begin{array}{c}
f_{2}\\
f_{\partial_{2}}
\end{array}\right),\left(\begin{array}{c}
e_{1}\\
e_{\partial_{1}}
\end{array}\right) \right\rangle_\mathscr{B}
\label{eq:SymmPairingBoundary}
\end{eqnarray}
where $\left(f_{i},\,f_{\partial_{i}},\,e_{i},\,e_{\partial_{i}}\right)\in\mathscr{B}; \quad i\in\left\{ 1,\,2\right\} $.
\end{proposition}

Before recalling the sketch of the proof, let us recall the general
definition of a Dirac structure. 
\begin{definition}
\cite{Courant88}\label{def:Dirac_structure} {[}Dirac structure{]}
Let $\mathcal{F}$ and $\mathcal{E}$ be two real vector spaces endowed with a non-degenerate bilinear form,
called \emph{pairing}, denoted by: 
\begin{equation}
\begin{array}{rcl}
\left\langle .|\,.\right\rangle :\,\mathcal{F}\times\mathcal{E} & \rightarrow & \mathbb{R}\\
\left(f,\,e\right) & \mapsto & \left\langle e,\,f\right\rangle 
\end{array}\label{eq:BilinearPairing}
\end{equation}

On the product space, called bond space $\mathcal{B}=\mathcal{F}\times\mathcal{E}$, consider the symmetric bilinear form, called \emph{plus pairing
}, defined as
\begin{equation}
\begin{array}{rcl}
\ll\cdot,\cdot\gg_{+}:\mathcal{B}\times\mathcal{B} & \rightarrow & \mathbb{R}\\
\left(\left(f_{1},\,e_{1}\right),\,\left(f,_{2}\,e_{2}\right)\right) & \mapsto & {\ll(f_{1},e_{1}),(f_{2},e_{2})\gg_{+}}:=\langle e_{1},\,f_{2}\rangle+\langle e_{2},\,f_{1}\rangle
\end{array}\label{eq:SymmetricProduct}
\end{equation}
A \emph{Dirac structure} is a subspace $\mathcal{D}\subset\mathcal{B}$
such that $\mathcal{D}=\mathcal{D}^{\bot}$, with $\perp$ denoting
the orthogonal complement with respect to the plus pairing $\ll,\gg_{+}$. 
\end{definition}

\begin{proof}
Let us first prove that $\mathscr{D}_{\mathcal{J}}\subset\mathscr{D}_{\mathcal{J}}^{\bot}$.
Recall \cite{schaftGeomPhys02} that this isotropy condition is equivalent to 
\begin{equation}
\ll(f,e),(f,e)\gg_{+}=0\qquad\forall\left(f,\,e\right)\in\mathcal{B}\label{eq:IsotropyCond_StokesDirac}
\end{equation}
In order to prove this, take any element $\left(\left(\begin{array}{c}
f\\
f_{\partial}
\end{array}\right),\left(\begin{array}{c}
e\\
e_{\partial}
\end{array}\right)\right)\in\mathscr{D}_{\mathcal{J}}$ and compute, using the definition of the Stokes-Dirac structure (\ref{eq:StokesDiracConstEq}),
 the definition of the port variables (\ref{eq:BoundaryPortVariablesSDS}),
the quadratic bilinear operator associated with the $2$-polynomial matrix (\ref{eq:2PolynomMatrix_J}),
the boundary term (\ref{eq:BoundaryTerm_J}) and the equation (\ref{eq:IntegrationbyParts})
\begin{eqnarray}
\ll(f,e),(f,e)\gg_{+} & = & 2\left(\left\langle f,e\right\rangle _{L^{2}}-\left\langle f_{\partial},e_{\partial}\right\rangle _{\mathbb{R}^{n}}\right)\label{eq:IsotropyCond_StokesDirac_expanded}\\
 & = & \int_{a}^{b}D_{\Phi}\left(e_{1},e_{2}\right)dz-\left[D_{\Psi}(v,w)(z)\right]_{a}^{b}\nonumber \\
 & = & 0\nonumber 
\end{eqnarray}
Secondly, let us briefly sketch the proof of the co-isotropy condition
$\mathscr{D}_{\mathcal{J}}^{\bot}\subset\mathscr{D}_{\mathcal{J}}$; see \cite{LeGorrecSIAM05} for details. Thereto,
take an element of the bond space 
$\left(\left(\begin{array}{c}
f_{1}\\
f_{\partial1}
\end{array}\right),\left(\begin{array}{c}
e_{1}\\
e_{\partial1}
\end{array}\right)\right)\in\mathcal{B}$ such that 
\[
\ll\left(\begin{array}{c}
f_{1}\\
f_{\partial1}
\end{array}\right),\left(\begin{array}{c}
e_{1}\\
e_{\partial1}
\end{array}\right),\,\left(\begin{array}{c}
f_{2}\\
f_{\partial2}
\end{array}\right),\left(\begin{array}{c}
e_{2}\\
e_{\partial2}
\end{array}\right)\gg_{+}=0
\]
for any $\left(\left(\begin{array}{c}
f_{2}\\
f_{\partial2}
\end{array}\right),\left(\begin{array}{c}
e_{2}\\
e_{\partial2}
\end{array}\right)\right)\in\mathscr{D}_{\mathcal{J}}$ . Then for any function $e_{2}$ vanishing at the boundary, according
to the definition of the symmetric pairing (\ref{eq:SymmPairingBoundary}),
and by the skew-symmetry of $\mathcal{J}$, 
\begin{eqnarray*}
\ll\left(\begin{array}{c}
f_{1}\\
f_{\partial1}
\end{array}\right),\left(\begin{array}{c}
e_{1}\\
e_{\partial1}
\end{array}\right),\,\left(\begin{array}{c}
f_{2}\\
f_{\partial2}
\end{array}\right),\left(\begin{array}{c}
e_{2}\\
e_{\partial2}
\end{array}\right)\gg_{+} & = & \left\langle f_{1},e_{2}\right\rangle _{L^{2}}+\left\langle \mathcal{J}e_{2},e_{1}\right\rangle _{L^{2}}\\
 & = & \left\langle \left(f_{1}-\mathcal{J}e_{1}\right),e_{2}\right\rangle _{L^{2}}
\end{eqnarray*}
for all $e_2$ vanishing at the boundary. Hence $f_{1}=\mathcal{J}e_{1}$. As a consequence, the symmetric pairing
is reduced to the boundary terms 
\begin{eqnarray*}
\ll\left(\begin{array}{c}
f_{1}\\
f_{\partial1}
\end{array}\right),\left(\begin{array}{c}
e_{1}\\
e_{\partial1}
\end{array}\right),\,\left(\begin{array}{c}
f_{2}\\
f_{\partial2}
\end{array}\right),\left(\begin{array}{c}
e_{2}\\
e_{\partial2}
\end{array}\right)\gg_{+} & = & -\left\langle f_{\partial1},e_{\partial2}\right\rangle _{\mathbb{R}^{n}}-\left\langle f_{\partial2},e_{\partial1}\right\rangle _{\mathbb{R}^{n}}\\
 & = &   -f_{\partial1}^{\top}e_{\partial2}-f_{\partial2}^{\top}e_{\partial1}\\
=  \left(\begin{array}{c}
f_{\partial1}\\
e_{\partial1}
\end{array}\right)^{\top}\frac{1}{\sqrt{2}}\left(\begin{array}{cc}
\Sigma & \Sigma\\
\Sigma & -\Sigma
\end{array}\right)\mathrm{tr}\left(T\left(\frac{\partial}{\partial z}\right)e_{2}\right) &  &
\end{eqnarray*}
which should vanish for any function $e_{2}$. By the assumption of
minimality, the matrices $T$ and $\left(\begin{array}{cc}
\Sigma & \Sigma\\
\Sigma & -\Sigma
\end{array}\right)$ have full rank and thus $\left(f_{\partial1},e_{\partial1}\right)$
satisfies the relation (\ref{eq:BoundaryPortVariablesSDS}) and hence
$\left(\left(\begin{array}{c}
f_{1}\\
f_{\partial1}
\end{array}\right),\left(\begin{array}{c}
e_{1}\\
e_{\partial1}
\end{array}\right)\right)\in\mathscr{D}_{\mathcal{J}}$ . 
\end{proof}

\subsubsection{Boundary port-Hamiltonian systems}

Boundary port-Hamiltonian systems as introduced in \cite{schaftGeomPhys02}
are \emph{open} Hamiltonian systems, defined with respect to Dirac
structures (see Proposition \ref{prop:StokesDiracStrucNorderLinearOrderOp}) as follows. 
\begin{definition}
\label{def:Linear_BPHS} \cite{schaftGeomPhys02} {[}Linear Boundary
Port Hamiltonian System{]} The linear boundary port-Hamiltonian system
on the state space $\mathscr{X}=L_{2}\left(a,\,b,\,\mathbb{R}^{n}\right)$
with respect to the Stokes-Dirac structure $\mathscr{D}_{\mathcal{J}}$
associated with the Hamiltonian matrix operator $\mathcal{J}$ as
in (\ref{eq:StokesDiracConstEq}) and generated by the functional
$\mathfrak{H}\left[x\right]$ in (\ref{eq:EnergyFunctionQuadratic}),
where $Q_{0}\in \mathbb{R}^{n\times n}$
is a positive semi-definite matrix, is the dynamical system 
\begin{equation}
\left(\left(\begin{array}{c}
\frac{\partial x}{\partial t}\\
f_{\partial}
\end{array}\right),\left(\begin{array}{c}
Q_{0}\,x\left(z\right)\\
e_{\partial}
\end{array}\right)\right)\in\mathscr{D}_{\mathcal{J}}\label{eq:BPHS_GeometricDef}
\end{equation}
where $\left(f_{\partial},\,e_{\partial}\right)$ are called \emph{boundary
port variables.} 
\end{definition}

In view of the definition of the Stokes-Dirac structure in (\ref{eq:StokesDiracConstEq}),
the linear port-Hamiltonian boundary system may also be formulated
as the system of equations 
\begin{eqnarray*}
\frac{\partial x}{\partial t} & = & \mathcal{J}(\frac{\partial}{\partial z})\,Q_{0}\,x\\
\left(\begin{array}{c}
f_{\partial}\\
e_{\partial}
\end{array}\right) & = & \frac{1}{\sqrt{2}}\left(\begin{array}{cc}
\Sigma & \Sigma\\
\Sigma & -\Sigma
\end{array}\right)\mathrm{tr}\left(T\left(\frac{\partial}{\partial z}\right)e\right)
\end{eqnarray*}

Furthermore, using the isotropy property of the Stokes-Dirac structure
(\ref{eq:IsotropyCond_StokesDirac}), one has $\left\langle f,e\right\rangle _{L^{2}}-\left\langle f_{\partial},e_{\partial}\right\rangle _{\mathbb{R}^{n}}=0$
, for any element $\left(\left(\begin{array}{c}
f\\
f_{\partial}
\end{array}\right),\left(\begin{array}{c}
e\\
e_{\partial}
\end{array}\right)\right)\in\mathscr{D}_{\mathcal{J}}$ . Hence a boundary port-Hamiltonian system satisfies the balance
equation 
\begin{equation}
\frac{d}{dt}\mathfrak{H}\left[x\right]=e_{\partial}^{\top}\,f_{\partial}\label{eq:EnergyBalance_BPHS}
\end{equation}

\begin{note}
Definition \ref{def:Linear_BPHS} could be generalized to the case where the matrix $Q_0$ is allowed to depend on the spatial variable $z$ and is, for instance, a coercive symmetric operator  $L^{2}\left(\left[a,\,b\right],\,\mathbb{R}^{n}\right) \rightarrow L^{2}\left(\left[a,\,b\right],\,\mathbb{R}^{n}\right)$ \cite{LeGorrecSIAM05,JacobZwart12}.
\end{note}

In physical systems modelling, the Hamiltonian is usually the total
energy of the system. In this case, \eqref{eq:EnergyBalance_BPHS}
is the energy balance equation which expresses that the rate of increase
of the total energy is equal to the power flow through its
boundary. Furthermore, since the matrix $Q_{0}$ is positive semi-definite
and hence the Hamiltonian is nonnegative, the system is \emph{impedance passive} \cite{LeGorrecSIAM05}.

These boundary port-Hamiltonian systems have proven to possess remarkable
properties \cite{JacobZwart12,Villegas_IEEE_TAC_09,Zwart_ESAIM_10,Jacob_IEEE_TAC19}
and have led to many results on the passivity-based boundary control
of distributed parameter systems controlled through their boundary
\cite{Macchelli_SIAM14,Macchelli_Automatica18_DissContrBPHS,Augner_SCL20_ExpStabBPHS}.

\subsection{\label{subsec:Motivational-examples} Motivational example: models
of elastic rods}

In this section, using the simple example of the elastic rod, we shall
recall how the boundary port-Hamiltonian model is established as the
prolongation of the classical symplectic Hamiltonian model. Furthermore,
we discuss its limitations with respect to the case of \emph{non-local}
elasticity constitutive relations.

\subsubsection{\label{subsec:Symplectic-formulation-of-rod} Symplectic formulation
of the dynamics of the elastic rod}

Let us recall the classical symplectic model of an elastic rod with
a local elasticity relation, obtained from variational principles and
the Legendre transformation of the Lagrangian with respect to the
velocity. Denote the \emph{displacement} of the elastic rod by $u(t,z)$
and its \emph{momentum density} by $p(t,z)$. The dynamical model
consists of a kinematic equation and a momentum balance equation,
combined into the symplectic Hamiltonian system 
\begin{equation}
\frac{\partial}{\partial t}\left(\begin{array}{c}
u\\
p
\end{array}\right)=\underbrace{\left(\begin{array}{cc}
0 & 1\\
-1 & 0
\end{array}\right)}_{=J_{0}}\;\left(\begin{array}{c}
\frac{\delta\mathfrak{H}}{\delta u}\\
\frac{\delta\mathfrak{H}}{\delta p}
\end{array}\right)\label{eq:VibrRod_StandardHamSyst}
\end{equation}
with Hamiltonian functional $\mathfrak{H}\left[u,\,p\right]=\mathfrak{K}\left(p\right)+\mathfrak{U}\left[u\right]$
\footnote{We follow the notation suggested in \cite[chap.7]{Olver93} and denote
by $G\left[u\right]=\int_{a}^{b}g\left(u,\,\frac{\partial u}{\partial z},\,..\,,\,\frac{\partial^{k}u}{\partial z^{k}}\right)\;dz$
a functional with density function $g$ depending on the variable $u\left(z\right)$ and a finite
number $k$ of its derivatives with respect to the spatial variable $z$.} consisting of the sum of the kinetic
energy $\mathfrak{K}\left(p\right)$ and the potential energy $\mathfrak{U}\left[u\right]$. The kinetic energy $\mathfrak{K}\left(p\right)$ is given as 
\begin{equation}
\mathfrak{K}(p)=\int_{a}^{b}\frac{1}{2}\frac{1}{\left(\rho A\right)}p^{2}\;dz\label{eq:VibrStringKinEn}
\end{equation}
where $\left(\rho A\right)$ denotes the linear mass density (product of
the section $A$ of the rod and the volume mass density $\rho$).
Furthermore, the potential energy is composed of two terms 
\begin{equation}
\mathfrak{U}\left[u\right]=\mathfrak{U}^{g}\left(u\right)+\int_{a}^{b}\frac{1}{2}\;T\;\left(\frac{\partial u}{\partial z}\right)^{2}\;dz,\label{eq:VibrStringPotEn_Local}
\end{equation}
where $\mathfrak{U}^{g}\left(u\right)$ is the elastic energy and the second term is the structural elastic energy (due to deformation of the rod).
For simplicity we will consider the elastic energy $\mathfrak{U}^{g}\left(u\right)$
to be the quadratic expression $\mathfrak{U}^{g}\left(u\right)=\int_{a}^{b}\frac{1}{2}ku^{2}\;dz$,
for some constant $k\geq0$. Note that the second term on the right-hand
side, the structural elastic energy, is expressed in terms of
the \emph{spatial derivative of the displacement}. By computing the
variational derivative $\frac{\delta\mathfrak{H}}{\delta u}\left(u\right)$ of the potential energy as indicated in (\ref{eq:VariationalDer_Polynomial}) one obtains the elastic
and structural elastic forces
\begin{equation}
\frac{\delta\mathfrak{H}}{\delta u}=\frac{\delta\mathfrak{U}}{\delta u}=k\,u-\frac{\partial}{\partial z}\left(T\;\frac{\partial u}{\partial z}\right)\label{eq:VibrStringVarDerPotEn}
\end{equation}
where $T>0$ denotes the \emph{elasticity modulus}.
It is clear that the symplectic model is \emph{not} a boundary port-Hamiltonian system since the structural elastic
energy, and hence the Hamiltonian functional, depends not on
the state variable $u$, but on its spatial derivative $\frac{\partial u}{\partial z}$. (This is in contrast with the Hamiltonian functional (\ref{eq:EnergyFunctionQuadratic}).)
Using the above expression of the elastic and structural elastic force, the dynamics (\ref{eq:VibrRod_StandardHamSyst})
becomes 
\begin{equation}
\frac{\partial}{\partial t}\left(\begin{array}{c}
u\\
p
\end{array}\right)=\underbrace{\left(\begin{array}{cc}
0 & 1\\
-1 & 0
\end{array}\right)}_{=J_{0}}\;\underbrace{\left(\begin{array}{cc}
k{\color{red}{\color{black}-\frac{\partial}{\partial z}T\frac{\partial}{\partial z}}} & 0\\
0 & \frac{1}{\rho A}
\end{array}\right)}_{=\mathcal{Q}(\frac{\partial}{\partial z})}\left(\begin{array}{c}
u\\
p
\end{array}\right)\label{eq:VibrRod_Sympl_DiffQ}
\end{equation}
where the dependency of the Hamiltonian on the spatial derivative
of the state variable $u$ implies that the energy is defined by a
matrix \emph{differential operator}.

\subsubsection{\label{subsec:Boundary-PH-formulation-Rod}Boundary port-Hamiltonian
formulation of the elastic rod}

On the other hand, the boundary port-Hamiltonian formulation is obtained
by augmenting the state space with the strain variable \cite{schaftGeomPhys02,maschkeFAP04}
\footnote{This procedure is actually a prolongation of the Hamiltonian system
to the first jet space. It has been also applied to the Timoshenko
beam model \cite{LeGorrecSIAM05}, the Mindlin and Kirchhoff plates
\cite{Brugnoli_AMM2019_I,Brugnoli_AMM2019_II}, some fluid dynamical
models \cite{MASCHKE_CPDE_2013}, and recently to the Allen-Cahn and
Cahn-Hilliard equations \cite{Vincent_IFAC_WC20_PhaseFields}.} 
\begin{equation}
\epsilon(t,z)=\frac{\partial u}{\partial z}(t,z)\label{eq:DefStrain}
\end{equation}
The total energy of the system is given by the Hamiltonian functional
\begin{equation}
\mathfrak{H}_{0}\left(u,\,\epsilon,\,p\right)=\mathfrak{U}^{g}\left(u\right)+\mathfrak{U}_{0}^{el}\left(\epsilon\right)+\mathfrak{K}(p)\label{eq:Hamiltonian0_VibratingRod}
\end{equation}
and the structural elastic energy $\mathfrak{U}_{0}^{el}\left(\epsilon\right)$
is thus expressed as a function of the state variable $\epsilon(t,z)$.

\begin{equation}
\mathfrak{U}_{0}^{el}(\epsilon)=\int_{a}^{b}\frac{1}{2}\;T\;\epsilon\left(z\right)^{2}\;dz\label{eq:VibrStringElEn}
\end{equation}
The variational derivative of the Hamiltonian (\ref{eq:Hamiltonian0_VibratingRod})
with respect to the state variables is written in vector notation
as 
\begin{equation}
\left(\begin{array}{c}
\delta_{u}H_{0}\\
\delta_{\epsilon}H_{0}\\
\delta_{p}H_{0}
\end{array}\right)=\underbrace{\left(\begin{array}{ccc}
k & 0 & 0\\
0 & T & 0\\
0 & 0 & \frac{1}{\left(\rho A\right)}
\end{array}\right)}_{=Q_{0}}\left(\begin{array}{c}
u\\
\epsilon\\
p
\end{array}\right)=\left(\begin{array}{c}
F^{pot}\\
\sigma\\
v
\end{array}\right)\label{eq:ElasticRodConstEqCoenergyVar}
\end{equation}
where $F^{pot}(t,z)$ is the potential force, $\sigma(t,z)$ the \emph{stress}
due to the structural elasticity, and $v(t,z)$ the \emph{velocity}.
The second line expresses the \emph{(local) elasticity relation} relating
the strain and the stress variables 
\begin{equation}
\sigma=T\,\epsilon\label{eq:LocalElasticityRelation}
\end{equation}
The dynamics of the elastic rod then can be formulated into an alternative
Hamiltonian formulation, consisting of the kinematic relation relating
the velocity to the momentum, the dynamics of the strain (obtained
by commutativity of the spatial and time derivatives) and the momentum
balance equation 
\begin{equation}
\frac{\partial}{\partial t}\left(\begin{array}{c}
u\\
\epsilon\\
p
\end{array}\right)=\underbrace{\left(\begin{array}{ccc}
0 & 0 & 1\\
0 & 0 & \frac{\partial}{\partial z}\\
-1 & \frac{\partial}{\partial z} & 0
\end{array}\right)}_{=\mathcal{J}_{1}}\,Q_{0}\left(\begin{array}{c}
u\\
\epsilon\\
p
\end{array}\right)\label{eq:HamSyst_VibrRod_Symmetry}
\end{equation}
generated by the Hamiltonian functional $\mathfrak{H}_{0}\left(u,\,\epsilon,\,p\right)$,
and with respect to the first-order matrix differential operator 
\begin{equation}
\mathcal{J}_{1}=P_{0}+P_{1}\frac{\partial}{\partial z}\qquad\textrm{with}\;P_{0}=\left(\begin{array}{ccc}
0 & 0 & 1\\
0 & 0 & 0\\
-1 & 0 & 0
\end{array}\right)\;\textrm{and}\;P_{1}=\left(\begin{array}{ccc}
0 & 0 & 0\\
0 & 0 & 1\\
0 & 1 & 0
\end{array}\right)\label{eq:ElasticRodHamOp_1order}
\end{equation}

In order to express the interaction of the system with its environment
through its boundary, this Hamiltonian system is augmented with a
pair of conjugated\emph{ boundary port variables} 
\begin{equation}
\left(\begin{array}{c}
f_{\partial}\\
e_{\partial}
\end{array}\right)=\left(\begin{array}{ccc}
0 & 0 & 1\\
0 & 1 & 0
\end{array}\right)\mathrm{tr}\,Q_{0}\left(\begin{array}{c}
u\\
\epsilon\\
p
\end{array}\right)=\mathrm{tr}\left(\begin{array}{c}
\frac{p}{\left(\rho A\left(z\right)\right)}\\
T\left(z\right)\;\epsilon\left(z\right)
\end{array}\right)\label{eq:BPvariables_VibrRod_Symmetry}
\end{equation}
which are the stress $\sigma=T\,\epsilon$ and velocity
$v$ at the boundary. The system (\ref{eq:HamSyst_VibrRod_Symmetry})
and (\ref{eq:BPvariables_VibrRod_Symmetry}) forms a \emph{boundary
port-Hamiltonian system} \cite{schaftGeomPhys02} with respect to
the Stokes-Dirac structure associated with the operator $\mathcal{J}_{1}$
according to Definition \ref{prop:StokesDiracStrucNorderLinearOrderOp}.

\subsubsection{\label{subsec:Descriptor-formulation-NonlocalRod}Descriptor formulation
of the elastic rod with nonlocal elasticity relations}

Finally, let us consider the elasto-dynamical model developed for
modelling carbon nanotubes \cite{Eringen_JAP1983} \cite{Karlicic_EJM2015},
which has been cast into a descriptor port-Hamiltonian formulation
in \cite{Zwart_MCMDS19_nonlocal}. For this composite material, the
constitutive relation of elasticity is \emph{nonlocal} and given by a \emph{differential}
relation between the strain $\epsilon$ and the stress $\sigma$ as
follows 
\begin{equation}
\sigma=T\left(1-\mu\frac{\partial^2}{\partial z^2}\right)^{-1}\,\epsilon,\label{eq:NonlocalElasticityRelation}
\end{equation}
where \emph{the parameter $\mu$ }and the elasticity modulus $T$ are positive real numbers.

These elasticity relations are expressed by the implicit relation
\begin{equation}
\left(1-\mu \frac{\partial^2}{\partial z^2}\right)\sigma-T\,\epsilon=0.\label{eq:ImplicitNonlocalElasticity}
\end{equation}
Alternatively, in image representation, the strain and the stress
variables are written as functions of a \emph{latent} strain $\lambda$
as follows 
\begin{equation}
\left(\begin{array}{c}
\epsilon\\
\sigma
\end{array}\right)=\left(\begin{array}{c}
1-\mu\left(\frac{\partial^2}{\partial z^2}\right)\\
T
\end{array}\right)\lambda\label{eq:ImageReprNonLocalElasticity}
\end{equation}

\begin{note}
The latent strain $\lambda =\left(1-\mu\frac{\partial^2}{\partial z^2}\right)^{-1}\epsilon$ corresponds to a non-local strain.
\end{note}

Inserting these relations into the system (\ref{eq:HamSyst_VibrRod_Symmetry}),
one obtains the descriptor system 
\begin{equation}
\frac{\partial}{\partial t}\underbrace{\left(\begin{array}{ccc}
1 & 0 & 0\\
0 & 1-\mu\left(\frac{\partial^2}{\partial z^2}\right) & 0\\
0 & 0 & 1
\end{array}\right)}_{=\mathcal{P}}\left(\begin{array}{c}
u\\
\lambda\\
p
\end{array}\right)=\underbrace{\left(\begin{array}{ccc}
0 & 0 & 1\\
0 & 0 & \frac{\partial}{\partial z}\\
-1 & \frac{\partial}{\partial z} & 0
\end{array}\right)}_{=\mathcal{J}\left(\frac{\partial}{\partial z}\right)}\;\underbrace{\left(\begin{array}{ccc}
k & 0 & 0\\
0 & T & 0\\
0 & 0 & \frac{1}{\rho A}
\end{array}\right)}_{=Q_{0}}\left(\begin{array}{c}
u\\
\lambda\\
p
\end{array}\right)\label{eq:HamSysNonlocalElasticityRelation}
\end{equation}
Note that, similarly to the symplectic formulation (\ref{eq:VibrRod_Sympl_DiffQ}),
the elastic constitutive relations are defined using a matrix differential
operator $\mathcal{P}$ . Furthermore the system (\ref{eq:HamSysNonlocalElasticityRelation})
is no more an explicit dynamical system; it is a descriptor system.
For this reason, this example does not fall into the class of boundary
port-Hamiltonian systems defined in Section \ref{subsec:Boundary-Port-HamiltonianSyst}.

\subsection{Summary}

Using the example of the elastic rod, we have recalled how the boundary
port-Hamiltonian formulation is obtained by augmenting the state space
in such a way that the Hamiltonian functional is only depending on
the state variables and not on their spatial derivatives. We have
noticed that both the symplectic formulation of the dynamics of the
elastic rod with local elasticity relation (\ref{eq:HamSyst_VibrRod_Symmetry}),
as well as the descriptor formulation arising from the nonlocal elasticity
relations (\ref{eq:HamSysNonlocalElasticityRelation}), do not fit
into Definition \ref{def:Linear_BPHS} of boundary port-Hamiltonian
systems. This is due to the fact that both the elasticity relations
(\ref{eq:VibrStringVarDerPotEn}) and (\ref{eq:ImplicitNonlocalElasticity})
are defined by \emph{differential relations}. The aim of the rest
of the paper is to \emph{extend} Definition \ref{def:Linear_BPHS}
of boundary port-Hamiltonian systems in order to encompass constitutive
relations of the energy that involve spatial derivatives of the state
variables.

\section{\label{sec:LagrangianSubspaces}Lagrangian subspaces associated with
reciprocal differential operators}

The aim of this section is to extend the definition of the energy
corresponding to the constitutive relation (\ref{eq:CoenergyExplicitLinear}),
defining the co-energy variable as a function of the state variable,
to relations involving differential operators as in (\ref{eq:VibrStringVarDerPotEn})
or their inverse as in (\ref{eq:NonlocalElasticityRelation}). This
will be done using the notion of a \emph{Lagrangian subspace} \cite{Weinstein_71},
which allows a generalization of the Hamiltonian functions of Hamiltonian
systems \cite[chap. 5.3]{Abraham_marsden87} \cite{Weinstein_Book_77}.
The same basic idea was recently employed for linear
and nonlinear finite-dimensional port-Hamiltonian systems in \cite{SCL_18,Schaft2020_DiracLagrangeNonlinear}.

\subsection{Lagrangian subspaces}

Let us briefly recall the definition of a Lagrangian subspace \cite{Libermann_marle87}. 
\begin{definition}
\label{def:Lagrangian-subspace} {[}Lagrangian subspace{]} Let $\mathcal{X}$
be a vector space (called \emph{space of energy variables}) and  $\mathcal{E}=\mathcal{X}^{*}$ 
its dual space (called \emph{space of effort variables}), endowed with the duality pairing $\left\langle \:|\:\right\rangle :\,\mathcal{X}\times\mathcal{E}\rightarrow\mathbb{R}$.
A \emph{Lagrangian subspace} in $\mathcal{X}\times\mathcal{E}$ \footnote{In fact we identify here the cotangent space $T^{*}\mathcal{X}$ of
the real vector space with the product space: $T^{*}\mathcal{X}\sim\mathcal{X}\times\mathcal{X}^{*}$
. For physical systems, it is called the \emph{Phase Space} of the
system.}, is a subspace $\mathcal{L}\subset\mathcal{X}\times\mathcal{E}$
such that 
\begin{equation}
\mathcal{L}=\mathcal{L}^{\bot_{-}}\label{eq:LagrangianCoIsotropy}
\end{equation}
where $\bot_{-}$ denotes the orthogonal with respect to the alternate
bilinear form on $\mathcal{X}\times\mathcal{E}$ 
\begin{equation}
\left\langle \left(x_{1},e_{1}\right),\left(x_{2},e_{2}\right)\right\rangle _{-}:=\left\langle e_{1}|x_{2}\right\rangle -\left\langle e_{2}|x_{1}\right\rangle \label{eq:AlternateBilinearForm}
\end{equation}
\end{definition}

Analogously to Dirac subspaces, which extend graphs of \emph{skew-symmetric}
mappings, Lagrangian subspaces extend graphs of \emph{symmetric} mappings.
For instance, the graph of the positive symmetric matrix operator
$\mathcal{Q}\in \mathbb{R}^{n\times n}$
\begin{equation}
\mathcal{L}=\left\{ \left(x,\,e\right)\in L^{2}\left(\left[a,\,b\right],\,\mathbb{R}^{n}\right)\;\mbox{ such that }\,e(z)=\mathcal{Q}\,x(z),\,\forall z\in\left[a,\,b\right]\right\} \label{eq:LagrangianSubspaceSymmetricMatrix}
\end{equation}
is a trivial example of a Lagrangian subspace. In physical systems
modelling, the map $\mathcal{Q}$ corresponds to the constitutive relation
between the energy variable $x$ and the co-energy variable $e$ of
the system, while the symmetry of the map
corresponds to Maxwell's reciprocity conditions \footnote{\label{fn:MaxwellReciprocity} 
\emph{Maxwell' reciprocity relations} are a property assumed for the constitutive relations
of the co-energy variables in the context of electrodynamics \cite[section 86]{Maxwell_1873_treatise},
elasticity \cite{Achenbach2003reciprocity}, or thermodynamics \cite[chap.7]{Callen85}}.

\subsection{Isotropy and reciprocal differential operators}

Let us consider the following class of \emph{constitutive relations}
defining the energy variables $x$ and the co-energy variables $e$
in $L^{2}\left(\left[a,\,b\right],\,\mathbb{R}^{n}\right)^{2}$ by  \eqref{eq:SubspaceConstRel-1} :
\begin{equation}
\left(\begin{array}{c}
x\\
e
\end{array}\right)=\left(\begin{array}{c}
\mathcal{P}\\
\mathcal{S}
\end{array}\right)\xi:=\mathcal{R}\xi\qquad\xi\in H_{0}^{M}\left(\left[a,\,b\right],\,\mathbb{R}^{n}\right)\label{eq:ConstRelEnergyImage}
\end{equation}
where $\mathcal{R}=\left(\begin{array}{c}
\mathcal{P}\\
\mathcal{S}
\end{array}\right)$, is a $\left(2n\times n\right)$ matrix differential operator, of
order $M$, composed of the two $\left(n\times n\right)$
matrix differential operators $\mathcal{P}$ and $\mathcal{S}$ mapping
the latent variable $\xi$ into the space of energy variables and
the space of co-energy variables respectively.

The isotropy condition (\ref{eq:LagrangianCoIsotropy}) leads to the
following condition on the operator $\mathcal{R}$. 
\begin{proposition}
\label{prop:Reciprocity} Consider the subspace associated with the
constitutive relation (\ref{eq:ConstRelEnergyImage}) 
\begin{eqnarray}
\mathcal{L}_{\mathcal{R}}^{0} & = & \left\{ \left(x,\,e\right)\in L^{2}\left(\left[a,\,b\right],\,\mathbb{R}^{n}\right)^{2}\: \mid \:\exists\xi\in H_{0}^{M}\left(\left[a,\,b\right],\,\mathbb{R}^{n}\right)\:\textrm{s.t.}\;\left(\begin{array}{c}
x\\
e
\end{array}\right)=\mathcal{R}\,\xi\right\} \label{eq:SubspaceConstRel}
\end{eqnarray}
where $H_{0}^{M}\left(\left[a,\,b\right],\,\mathbb{R}^{n\times n}\right)$
denotes the set of functions with support contained in $\left]a,b\right[$.
The isotropy condition $\mathcal{L}_{\mathcal{R}}^{0}\subset\left(\mathcal{L}_{\mathcal{R}}^{0}\right)^{\bot}$
with respect to the alternate form (\ref{eq:AlternateBilinearForm})
associated with the $L^2$ inner product, is satisfied if and
only if the following \emph{reciprocity condition is satisfied} \footnote{We call this condition \emph{reciprocity condition} in reference to
\emph{Maxwell' reciprocity relations}.} 

\begin{equation}
\mathcal{S}^{*}\mathcal{P}-\mathcal{P}^{*}\mathcal{S}=\mathcal{R}^{*}\Theta_{n}\mathcal{R}=0\label{eq:ReciprocityConditionOperators}
\end{equation}
denoting 
\begin{equation}
\Theta_{n}=\left(\begin{array}{cc}
0_{n} & -I_{n}\\
I_{n} & 0_{n}
\end{array}\right)\label{eq:SymplecticMatrix}
\end{equation}
the symplectic matrix of order $2n$ and $\mathcal{R}^{*}$ the adjoint
of the operator $\mathcal{R}$ with respect to $L^{2}\left(\left[a,\,b\right]\right)$. 
\end{proposition}

\begin{proof}
Consider the state and energy spaces $\mathcal{X}=\mathcal{E}=L^{2}\left(\left[a,\,b\right],\,\mathbb{R}^{n\times n}\right)$
with the pairing $\left\langle \:|\:\right\rangle $ being the usual
$L^2$ inner product
\begin{equation}
\left\langle e|x\right\rangle =\int_{a}^{b}e^{\top}\left(z\right)\,x\left(z\right)dz\label{eq:HilbertSpaceProduct}
\end{equation}

Then the isotropy condition $\mathcal{L}_{\mathcal{R}}^{0}\subset\left(\mathcal{L}_{\mathcal{R}}^{0}\right)^{\bot}$
\begin{equation}
\left\langle \left(\begin{array}{c}
x_{1}\\
e_{1}
\end{array}\right),\,\left(\begin{array}{c}
x_{2}\\
e_{2}
\end{array}\right)\right\rangle _{-}:=\left\langle e_{1}|x_{2}\right\rangle -\left\langle e_{2}|x_{1}\right\rangle =0\quad\forall\left(\begin{array}{c}
x_{i}\\
e_{i}
\end{array}\right)\in\mathcal{L}_{\mathcal{R}}^{0},\,i=1,2\label{eq:IsotropyCondition}
\end{equation}
may be written, using (\ref{eq:ConstRelEnergyImage}), for functions
$\xi_{1},\,\xi_{2}\in H_{0}^{M}\left(\left[a,\,b\right],\,\mathbb{R}^{n}\right)$,
as follows 
\begin{eqnarray}
\left\langle \left(\begin{array}{c}
x_{1}\\
e_{1}
\end{array}\right),\,\left(\begin{array}{c}
x_{2}\\
e_{2}
\end{array}\right)\right\rangle _{-} & = & \int_{a}^{b}e_{1}^{\top}\left(z\right)\,x_{2}\left(z\right)dz-\int_{a}^{b}e_{2}^{\top}\left(z\right)\,x_{1}\left(z\right)dz\nonumber \\
 & = & \int_{a}^{b}\left[\left(\mathcal{S}\xi_{1}\left(z\right)\right)^{\top}\,\mathcal{P}\xi_{2}\left(z\right)-\left(\mathcal{S}\xi_{2}\left(z\right)\right)^{\top}\,\mathcal{P}\xi_{1}\left(z\right)\right]dz\nonumber \\
 & = & \int_{a}^{b}\left[\left(\xi_{1}\left(z\right)\right)^{\top}\,\left(\mathcal{S}^{*}\mathcal{P}-\mathcal{P}^{*}\mathcal{S}\right)\,\xi_{2}\left(z\right)\right]dz\label{eq:AltProdHN_0}
\end{eqnarray}
The product is zero for any pairs $\xi_{1},\,\xi_{2}\in H_{0}^{M}\left(\left[a,\,b\right],\,\mathbb{R}^{n}\right)$
if and only if the reciprocity condition (\ref{eq:ReciprocityConditionOperators})
is satisfied. 
\end{proof}
In this paper we only consider \emph{constant coefficient }matrix
differential operators of the type 
\begin{equation}
\mathcal{R}=\sum_{i=0}^{M}R_{i}\frac{\partial^{i}}{\partial z^{i}}:= R\left(\frac{\partial}{\partial z}\right)\:,\;M\in\mathbb{N},\;R_{i}\in\mathbb{R}^{2n\times n}\label{eq:MatrixOperatorConstRelEnergy}
\end{equation}
where $R\left(s\right)=\sum_{i=0}^{M}R_{i}s^{i}$ is \emph{a polynomial
matrix} of order $M$. In view of the constitutive relation (\ref{eq:ConstRelEnergyImage}),
the polynomial matrix is decomposed into two $\left(n\times n\right)$
polynomial matrices $P\left(s\right)$ and $S\left(s\right)$ 
\begin{equation}
R\left(s\right)=\left(\begin{array}{c}
P\left(s\right)\\
S\left(s\right)
\end{array}\right)\label{eq:PolynomialMatrixDecomposition}
\end{equation}

Let us now express the reciprocity condition (\ref{eq:ReciprocityConditionOperators})
in terms of the polynomial matrices $P\left(s\right)$ and $S\left(s\right)$. 
\begin{proposition}
The reciprocity condition (\ref{eq:ReciprocityConditionOperators})
on the differential operator $\mathcal{R}$ (\ref{eq:MatrixOperatorConstRelEnergy})
is equivalent to the following condition on the associated polynomial
matrix (\ref{eq:PolynomialMatrixDecomposition}) 
\begin{equation}
R^{\top}\left(-s\right)\Theta_{n}R\left(s\right)=S^{\top}\left(-s\right)P\left(s\right)-P^{\top}\left(-s\right)S\left(s\right)=0\label{eq:ReciprocityConditionsPolynomialMatrix}
\end{equation}
\end{proposition}

Note that if $P(s)$ is the identity matrix, then the above condition
reduces to $S\left(s\right)=S^{\top}\left(-s\right)$. That is, the
matrix differential operator $S\left(\frac{\partial}{\partial z}\right)$
is \emph{formally self-adjoint}.

As an illustration, let us express the reciprocity condition in a
number of particular cases. Defining the polynomial matrices $P\left(s\right)=\sum_{i=0}^{M}P_{i}s^{i}$
and $S\left(s\right)=\sum_{i=0}^{M}S_{i}s^{i}$ , consider first the
case when the reciprocal operator $\mathcal{R}$ in (\ref{eq:MatrixOperatorConstRelEnergy})
is of order zero. 
\begin{corr}
Consider an operator $\mathcal{R}=R_{0}=\left(\begin{array}{c}
P_{0}\\
S_{0}
\end{array}\right)$ where $P_{0},Q_{0}\in\mathbb{R}^{n\times n}$, being of differential
order $0$ , the reciprocity condition (\ref{eq:ReciprocityConditionsPolynomialMatrix})
is equivalent to 
\begin{equation}
R_{0}^{\top}\Theta_{n}R_{0}=S_{0}^{\top}P_{0}-P_{0}^{\top}S_{0}=0\label{eq:ReciprocityCondOperator0}
\end{equation}
\end{corr}

In other words, for reciprocal operators of order zero, \emph{the
matrix $S_{0}^{\top}P_{0}$ is symmetric}, which is the condition
obtained in the finite-dimensional case \cite{SCL_18} . 
\begin{example}
\label{exa:The-elastic-rod_Thermo_BPHS-1} Consider again the example
of the elastic rod with local elasticity relations, energy variables
$x_{0}\left(t,z\right)^{\top}=\left(u,\,\epsilon\,,p\right)$ and
co-energy variables $e_{0}\left(t,z\right)^{\top}=\left(F^{pot},\,\sigma\,,v\right)$
defined in (\ref{eq:ElasticRodConstEqCoenergyVar}). Equivalently,
one may as well define the energy and co-energy variables using the
latent variable $\xi_{0}\left(t,z\right)^{\top}=\left(u,\,\epsilon\,,v\right)$
where the velocity $v$ is used instead of the momentum $p$. This
leads to the following formulation 
\begin{equation}
\frac{\partial}{\partial t}\underbrace{\left(\begin{array}{ccc}
1 & 0 & 0\\
0 & 1 & 0\\
0 & 0 & \left(\rho A\right)
\end{array}\right)}_{=P_{0}}\left(\begin{array}{c}
u\\
\epsilon\\
v
\end{array}\right)=\underbrace{\left(\begin{array}{ccc}
0 & 0 & 1\\
0 & 0 & \frac{\partial}{\partial z}\\
-1 & \frac{\partial}{\partial z} & 0
\end{array}\right)}_{=\mathcal{J}_{1}}\,\underbrace{\left(\begin{array}{ccc}
k & 0 & 0\\
0 & T & 0\\
0 & 0 & 1
\end{array}\right)}_{=S_{0}}\left(\begin{array}{c}
u\\
\epsilon\\
v
\end{array}\right)\label{eq:Port_HamSyst_VibrRod_Symmetry-bis-1}
\end{equation}
\end{example}

For second-order differential operators, as appearing in the example
of the elastic rod, the reciprocity condition is obtained by a straightforward
calculation as follows. 
\begin{corr} \label{corr:ReciprocityCond2ndOrderOperators}
For second-order $n\times n$ matrix differential operators: 
\begin{equation}
\mathcal{P}=P_{0}+P_{1}\frac{\partial}{\partial z}+P_{2}\frac{\partial^{2}}{\partial z^{2}}\qquad\textrm{and}\quad\mathcal{S}=S_{0}+S_{1}\frac{\partial}{\partial z}+S_{2}\frac{\partial^{2}}{\partial z^{2}}\label{eq:SecondOrderSymOp}
\end{equation}
the reciprocity condition (\ref{eq:ReciprocityConditionsPolynomialMatrix})
is equivalent to: $S_{0}^{\top}P_{0}$ and $S_{2}^{\top}P_{2}$ are
symmetric, $\left(P_{0}^{\top}S_{1}-S_{0}^{\top}P_{1}\right)$ and
$\left(P_{1}^{\top}S_{2}-S_{1}^{\top}P_{2}\right)$ are skew-symmetric,
$\left(P_{0}^{\top}S_{2}-S_{0}^{\top}P_{2}+S_{1}^{\top}P_{1}\right)$
is symmetric. 
\end{corr}

These conditions are immediately checked on the examples of the symplectic
formulation of the elastic rod (\ref{eq:VibrRod_Sympl_DiffQ}) and
of the rod with non-local elasticity relations (\ref{eq:HamSysNonlocalElasticityRelation}).

\begin{example}
\label{exa:The-elastic-rod_symplectic_ReciprocOperators}
Consider the example of the elastic rod with local elasticity relations as presented in Section \ref{subsec:Symplectic-formulation-of-rod}. The latent variable is the vector composed of the displacement $u$ of the rod and the momentum density $p$
$$\xi=\left(\begin{array}{c}
u\\
p
\end{array}\right)$$ The coenergy variables are defined by the matrix $\mathcal{Q}\left(\frac{\partial}{\partial z}\right)$
 in (\ref{eq:VibrRod_Sympl_DiffQ}). Hence the second order reciprocal operator defining the phase space variables (\ref{eq:SecondOrderSymOp}) is defined by the matrices $P_0 =I_{2}$,  $P_1 =P_2=0_{2}$ and $S_0=\left(\begin{array}{cc}
k & 0\\
0 & \frac{1}{\left(\rho A\right)}
\end{array}\right)$,
$S_1 =0_{2}$
and
$S_2=\left(\begin{array}{cc}
-T & 0\\
0 & 0
\end{array}\right)$ which obviously satisfy the reciprocity conditions of Corollary \ref{corr:ReciprocityCond2ndOrderOperators}.
\end{example}

\subsection{Lagrangian subspaces associated with maximally reciprocal operators}

Generalizing the image representation of Lagrangian vector subspaces
in the finite dimensional case \cite{SCL_18}, we shall define a class
of Lagrangian subspaces associated with reciprocal $\left(2n\times n\right)$
matrix differential operators $\mathcal{R}$. For this purpose we
shall consider a subclass of reciprocal operators satisfying a maximality
condition. 
\begin{definition}
A matrix $\left(2n\times n\right)$ differential operator $\mathcal{R}$
(\ref{eq:MatrixOperatorConstRelEnergy}) is \emph{maximally reciprocal}
if it satisfies the reciprocity condition (\ref{eq:ReciprocityConditionOperators})
and the maximality condition 
\begin{equation}
\textrm{Im}\,\mathcal{R}=\textrm{ker}\, \mathcal{R}^{*}\Theta_{n}\label{eq:MaximalityConditionOperator}
\end{equation}
Equivalently, the associated polynomial matrix $R\left(s\right)$
satisfies 
\begin{equation}
\textrm{rank}\,R\left(s\right)=n\quad\forall s\in\mathbb{C}\label{eq:MaximalityConditionPolynomialMatrix}
\end{equation}
\end{definition}

\begin{note}
Note that from the reciprocity condition (\ref{eq:ReciprocityConditionOperators}), it follows that $\textrm{Im}\,\mathcal{R}\subset\textrm{ker}\, \mathcal{R}^{*}\Theta_{n}$. Hence, for reciprocal operators, the maximality condition (\ref{eq:MaximalityConditionOperator}) is equivalent to $\textrm{Im}\,\mathcal{R}\supset\textrm{ker}\, \mathcal{R}^{*}\Theta_{n}$.
\end{note}

With this maximality condition, the graph of the constitutive relations
(\ref{eq:ConstRelEnergyImage}) defines a Lagrangian subspace. 
\begin{proposition}
\label{prop:LagrangeSubspaceDiff_Homogeneous} Consider a $\left(2n\times n\right)$
maximally reciprocal matrix differential operator $\mathcal{R}$ (\ref{eq:MatrixOperatorConstRelEnergy}),
then the subspace $\mathcal{L}_{\mathcal{R}}^{0}$ of the constitutive
relations (\ref{eq:SubspaceConstRel}) is a Lagrangian subspace with
respect to the alternate form (\ref{eq:AlternateBilinearForm}), where
the state and energy spaces are identified as $\mathcal{X}=\mathcal{E}=L_{2}\left(\left[a,\,b\right],\,\mathbb{R}^{n\times n}\right)$
and the pairing $\left\langle \:|\:\right\rangle $ is the usual $L_2$ inner product. 
\end{proposition}

\begin{proof}
\begin{flushleft}
The isotropy condition $\mathcal{L}_{\mathcal{R}}^{0}\subset\left(\mathcal{L}_{\mathcal{R}}^{0}\right)^{\bot}$
is satisfied in view of the reciprocity of the operator $\mathcal{R}$
and Proposition \ref{prop:Reciprocity}. Now consider the co-isotropy
condition $\left(\mathcal{L}_{\mathcal{R}}^{0}\right)^{\bot}\subset\mathcal{L}_{\mathcal{R}}^{0}$
. Consider an element $\left(x_{2},\,e_{2}\right)\in L^{2}\left(\left[a,\,b\right],\,\mathbb{R}^{n}\right)^{2}$
such that the product (\ref{eq:AlternateBilinearForm}) with any $\left(x_{1},\,e_{1}\right)\in\mathcal{L}_{\mathcal{R}}^{0}$
vanishes. Then for $\xi_{1}\in H_{0}^{M}\left(a,\,b,\,\mathbb{R}^{n}\right)$,
the alternate product is computed as 

\begin{eqnarray*}
\left\langle e_{1}|x_{2}\right\rangle -\left\langle e_{2}|x_{1}\right\rangle  & = & \int_{a}^{b}\left(e_{1}^{\top}\,x_{2}-e_{2}^{\top}\,x_{1}\right)\;dz\\
 & = & \int_{a}^{b}\left(\left(\mathcal{S}\xi_{1}\right)^{\top}\,x_{2}-e_{2}^{\top}\,\left(\mathcal{P}\xi_{1}\right)\right)\;dz\\
 & = & \int_{a}^{b} \left(\mathcal{R}\xi_{1}\right)^{\top} \Theta_n \left(\begin{array}{c}
x_{2}\\
e_{2}
\end{array}\right)  \;dz
\end{eqnarray*}
Being $0$ for all $\xi_{1}\in H_{0}^{N}\left(a,\,b,\,\mathbb{R}^{n}\right)$
, implies  $\left(\begin{array}{c}
x_{2}\\
e_{2}
\end{array}\right)\in\ker\mathcal{R}^{*}\Theta_n$ and by the maximality condition (\ref{eq:MaximalityConditionOperator}),
then $\left(\begin{array}{c}
x_{2}\\
e_{2}
\end{array}\right)\in\textrm{Im}\,\mathcal{R}$.
Thus $\left(\mathcal{L}_{\mathcal{R}}^{0}\right)^{\bot}\subset\mathcal{L}_{\mathcal{R}}^{0}$,
and the subspace $\mathcal{L}_{\mathcal{R}}^{0}$ is Lagrangian. 

\par\end{flushleft}
\end{proof}

\subsection{Stokes-Lagrange subspaces associated with maximally reciprocal operators}

When the energy and effort variables are not vanishing on the boundary,
then the graph $\mathcal{L}_{\mathcal{R}}^{0}$ of the constitutive
relations (\ref{eq:ConstRelEnergyImage}) is no more a Lagrangian
subspace. This means, analogously to the Stokes-Dirac structure case,
that the reciprocity conditions have to be \emph{extended} by including
boundary variables.

Define the \emph{bilinear differential operator}, denoted by $B_{\bar{\Phi}}$,
and such that for any pair $\left(\xi_{1},\,\xi_{2}\right)$ of elements
in $H^{M}\left(a,\,b,\,\mathbb{R}^{n}\right)$  
\begin{eqnarray}
B_{\bar{\Phi}}\left(\xi_{1},\,\xi_{2}\right) & = & \xi_{1}^{\top}\left[R^{\top}\left(-\frac{\partial}{\partial z}\right)\Theta_{n}R\left(\frac{\partial}{\partial z}\right)\right]\xi_{2}\label{eq:BilinearDiffOperator}\\
 & = & \xi_{1}^{\top}\left[S^{\top}\left(-\frac{\partial}{\partial z}\right)P\left(\frac{\partial}{\partial z}\right)-P^{\top}\left(-\frac{\partial}{\partial z}\right)S\left(\frac{\partial}{\partial z}\right)\right]\xi_{2}
\end{eqnarray}
and denote by $\bar{\Phi}_{k,l}\in\mathbb{R}^{2n\times2n}$ its \emph{coefficient
matrices.} 

Recalling the expression of the bilinear operator (\ref{eq:BilinearDiffOperator}),
a direct calculation gives the expression of the two-variable polynomial
matrix $\bar{\Phi}\left(\zeta,\mu\right)$, in terms of the polynomial matrix
$R\left(s\right)$ associated with the reciprocal operator $\mathcal{R}$
as follows 
\begin{equation}
\bar{\Phi}\left(\zeta,\eta\right)=R^{\top}\left(\zeta\right)\Theta_{n}R\left(\eta\right)\label{eq:2VarPolynomialMatrix_R}
\end{equation}

Define its derivative, denoted by $B_{\bar{\Psi}}\left(\xi_{1},\,\xi_{2}\right)$,
according to the definition (\ref{eq:DerivativeDiffOperator}), and
denote its coefficient matrix by $\bar{\Psi}\left(\zeta,\eta\right)$. Note
that this matrix is skew-symmetric as a consequence of the skew-symmetry
of $\Theta_{n}$ . Furthermore it admits a canonical symplectic decomposition
as stated in the following theorem. 
\begin{proposition}
\label{prop:BoundaryPolynomialMatrix}There exist an integer $p$
and a $\left(2p\times n\right)$ polynomial matrix $R_{\partial}\left(s\right)$
of full row-rank, which defines the following canonical factorization
of the two-variable polynomial matrix $\bar{\Psi}\left(\zeta,\eta\right)$
\\

\begin{equation}
\bar{\Psi}\left(\zeta,\eta\right)=R_{\partial}^{\top}\left(\zeta\right)\Theta_{p}R_{\partial}\left(\eta\right)\label{eq:FactorBoundary2Polynomial}
\end{equation}
where $\Theta_{p}$ denotes the symplectic matrix (\ref{eq:SymplecticMatrix})
of order $2p$. 
\end{proposition}

\begin{proof}
This follows by slightly adapting the proof as given in \cite[Sec.3]{Willems_SIAM98}\cite[Sec 2.3]{Schaft_SIAM11}
and using the skew-symmetry of $\bar{\Psi}\left(\zeta,\eta\right)$. 
\end{proof}
Using the polynomial matrix $R\left(s\right)$ and (\ref{eq:BilinearDiffOperator}),
the formula (\ref{eq:IntegrationbyParts}) may be written 
\begin{equation}
\xi_{1}^{\top}\left[R^{\top}\left(-\frac{\partial}{\partial z}\right)\Theta_{n}R\left(\frac{\partial}{\partial z}\right)\right]\xi_{2}=\frac{\partial}{\partial z}\left(\left(R_{\partial}\left(\frac{\partial}{\partial z}\right)\xi_{1}\right)^{\top}\Theta_{p}R_{\partial}\left(\frac{\partial}{\partial z}\right)\xi_{2}\right)\label{eq:IntByParts_BoundaryEnergyVar}
\end{equation}
Using the natural decomposition, induced by the symplectic matrix
$\Theta_{p}$ , of the boundary polynomial matrix into two $\left(p\times n\right)$
polynomial matrices 
\begin{equation}
R_{\partial}\left(s\right)=\left(\begin{array}{c}
P_{\partial}\left(s\right)\\
S_{\partial}\left(s\right)
\end{array}\right)\label{eq:BoundaryMapRecip_Decomposed}
\end{equation}
the preceding equality may also be written as 
\begin{equation}
S^{T}(\zeta)P(\eta)-P^{T}(\zeta)S(\eta)=(\zeta+\eta)\big[S_{\partial}^{T}(\zeta)P_{\partial}(\eta)-P_{\partial}^{T}(\zeta)S_{\partial}(\eta)\big]\label{eq:FundamentalRelation_BoundaryMap_StokesLagrange}
\end{equation}

We shall use this boundary operator in order to generalize the Lagrangian
subspaces associated with reciprocal operators for functions that
are not vanishing on the boundary. 
\begin{proposition}
\label{prop:StokesLagrangeSubspace}{[}Stokes-Lagrange subspace{]}
Consider the extended state space $\mathcal{C}=L^{2}\left(\left[a,\,b\right],\,\mathbb{R}^{n}\right)^{2}\times\:\mathbb{R}^{p}\times\mathbb{R}^{p}$
. The subspace associated with a maximally reciprocal operator $\mathcal{R}$
(\ref{eq:MatrixOperatorConstRelEnergy}) 
\begin{eqnarray}
\mathcal{L}_{\mathcal{R}} & = & \left\{ \left(x,\,e,\chi_{\partial},\,\varepsilon_{\partial}\right)\in\mathcal{C}/\:\exists\xi\in H^{M}\left(\left[a,\,b\right]\,\mathbb{R}^{n}\right)\:\textrm{s.t.}\;\right.\label{eq:SubspaceConstRel-1}\\
 &  & \left.\left(\begin{array}{c}
x\\
e
\end{array}\right)=\mathcal{R}\,\xi;\quad\left(\begin{array}{c}
\chi_{\partial}\\
\varepsilon_{\partial}
\end{array}\right)=\mathrm{tr}\left(\left(\begin{array}{c}
P_{\partial}(\frac{\partial}{\partial z})\\
S_{\partial}(\frac{\partial}{\partial z})
\end{array}\right)\xi\right)\right\} \label{eq:EnergyBoundaryPortVariables}
\end{eqnarray}
where $\mathrm{tr}$ denotes the trace operator (\ref{eq:TraceOperatorDef}) and the polynomial
matrices $P_{\partial}(s)$ and $S_{\partial}(s)$ are defined in Proposition \ref{prop:BoundaryPolynomialMatrix},
is a Lagrangian subspace of the extended state space $\mathcal{C}$
with respect to the alternate product for any two pairs of elements
$c_{i}=\left(x_{i},\,e_{i},\chi_{\partial i},\,\varepsilon_{\partial i}\right)$,
$i\in\left\{ 1,2\right\} $ 
\begin{equation}
\left\llbracket c_{1},c_{2}\right\rrbracket =\left\langle e_{1}|x_{2}\right\rangle -\left\langle e_{2}|x_{1}\right\rangle -\left[\chi_{\partial1}^{\top}\varepsilon_{\partial2}-\chi_{\partial2}^{\top}\varepsilon_{\partial1}\right]_{a}^{b}\label{eq:AlternateProductWithBoundary}
\end{equation}
\end{proposition}

\begin{proof}
\begin{flushleft}
The isotropy condition $\mathcal{L}_{\mathcal{R}}\subset\left(\mathcal{L}_{\mathcal{R}}\right)^{\bot}$
is satisfied in view of the integration by parts formula (\ref{eq:IntByParts_BoundaryEnergyVar}).
Now consider the co-isotropy condition $\left(\mathcal{L}_{\mathcal{R}}\right)^{\bot}\subset\mathcal{L}_{\mathcal{R}}$
. Consider any element $c_{2}=\left(x_{2},\,e_{2},\chi_{\partial2},\,\varepsilon_{\partial2}\right)\in\mathcal{C}$
such that the product (\ref{eq:AlternateBilinearForm}) with any $c_{1}=\left(x_{1},\,e_{1},\chi_{\partial1},\,\varepsilon_{\partial1}\right)\in\mathcal{L}_{\mathcal{R}}$
vanishes. Then for some $\xi_{1}\in H^{M}\left(\left[a,\,b\right],\,\mathbb{R}^{n}\right)$,
the alternate product may be computed as 
\begin{eqnarray*}
\left\llbracket c_{1},c_{2}\right\rrbracket  & = & \left\langle e_{1}|x_{2}\right\rangle -\left\langle e_{2}|x_{1}\right\rangle -\left[\chi_{\partial1}^{\top}\varepsilon_{\partial2}-\chi_{\partial2}^{\top}\varepsilon_{\partial1}\right]_{a}^{b}\\
 & = & \int_{a}^{b}\left(\left(R\left(\frac{\partial}{\partial z}\right)\xi_{1}\right)^{\top}\Theta_{n}\left(\begin{array}{c}
x_{2}\\
e_{2}
\end{array}\right)\right)\;dz-\left[\left(R_{\partial}\left(\frac{\partial}{\partial z}\right)\xi_{1}\right)^{\top}\Theta_{p}\left(\begin{array}{c}
\chi_{\partial2}\\
\varepsilon_{\partial2}
\end{array}\right)\right]_{a}^{b}
\end{eqnarray*}
Consider first $\xi_{1}\in H_{0}^{N}\left(\left[a,\,b\right],\,\mathbb{R}^{n}\right)$. Then the second term in the alternate product vanishes and we may follow
 the proof of Proposition \ref{prop:LagrangeSubspaceDiff_Homogeneous},
 and conclude that, since the operator $\mathcal{R}$ is assumed to be maximally reciprocal,
there exists an element $\xi_{2}\in H^{N}\left(a,\,b,\,\mathbb{R}^{n}\right)$
that generates $\left(\begin{array}{c}
x_{2}\\
e_{2}
\end{array}\right)$ according to (\ref{eq:ConstRelEnergyImage}) . As a result, using
(\ref{eq:IntByParts_BoundaryEnergyVar}), the alternate product becomes
\begin{eqnarray*}
\left\llbracket c_{1},c_{2}\right\rrbracket  & = & \int_{a}^{b}\left(\left(R\left(\frac{\partial}{\partial z}\right)\xi_{1}\right)^{\top}\Theta_{n}\left(R\left(\frac{\partial}{\partial z}\right)\xi_{2}\right)\right)\;dz\\
&  &  -\left[\left(R_{\partial}\left(\frac{\partial}{\partial z}\right)\xi_{1}\right)^{\top}\Theta_{p}\left(\begin{array}{c}
\chi_{\partial2}\\
\varepsilon_{\partial2}
\end{array}\right)\right]_{a}^{b}\\
 & = & \left[\left(R_{\partial}\left(\frac{\partial}{\partial z}\right)\xi_{1}\right)^{\top}\Theta_{p}\left(R_{\partial}\left(\frac{\partial}{\partial z}\right)\xi_{2}-\left(\begin{array}{c}
\chi_{\partial2}\\
\varepsilon_{\partial2}
\end{array}\right)\right)\right]_{a}^{b}
\end{eqnarray*}
Since $R_{\partial}\left(s\right)$ defines a minimal factorization of
the boundary 2-polynomial $\bar{\Psi}\left(\zeta,\mu\right)$ and has full
row rank characteristic matrix, the image of operator $R_{\partial}\left(\frac{\partial}{\partial z}\right)$
at the boundaries $\left\{ a,b\right\} $ is $\mathbb{R}^{2p}$. If
$\left\llbracket c_{1},c_{2}\right\rrbracket =0$ for any $\xi_{1}\in H^{N}\left(a,\,b,\,\mathbb{R}^{n}\right)$
, since the symplectic matrix $\Theta_{p}$ is full-rank, $R_{\partial}\left(\frac{\partial}{\partial z}\right)\xi_{2}-\left(\begin{array}{c}
\chi_{\partial2}\\
\varepsilon_{\partial2}
\end{array}\right)=0$. Hence $c_{2}$ satisfies (\ref{eq:SubspaceConstRel-1}) and (\ref{eq:EnergyBoundaryPortVariables})
and $c_{2}\in\mathcal{L}_{\mathcal{R}}$. This concludes the proof. 
\par\end{flushleft}

\end{proof}
In the sequel, the variables defined by the boundary operator $R_{\partial}\left(s\right)$
in (\ref{eq:EnergyBoundaryPortVariables}) will be called the \emph{energy
boundary port variables}. Indeed, for physical systems models, the
product of the energy boundary port variables of the Stokes-Lagrange
subspace has dimension of \emph{energy}, contrary to the product
of the port variables of the Stokes-Dirac structure which has dimension
of \emph{power}. 
\begin{example}
\label{exa:Rod_EnergyBoundaryVariablesDiffQ} Consider again the symplectic
formulation of the flexible rod with local elasticity relation presented in
Section \ref{subsec:Symplectic-formulation-of-rod}. Recall that, cf. Example \ref{exa:The-elastic-rod_symplectic_ReciprocOperators}, the vector of latent variables is 
 $\xi=\left(\begin{array}{c}
u\\
p
\end{array}\right)$ where $u$ is the displacement a point of the rod and $p$ the momentum density and
the polynomial matrix, associated with the reciprocal operator, is
$R\left(s\right)=\left(\begin{array}{c}
P\left(s\right)\\
S\left(s\right)
\end{array}\right)$ with $P\left(s\right)=I_{2}$ and $S\left(s\right)=\left(\begin{array}{cc}
\left(k-T\,s^{2}\right) & 0\\
0 & \frac{1}{\left(\rho A\right)}
\end{array}\right)$. From the two-variable polynomial matrix (\ref{eq:2VarPolynomialMatrix_R})
associated with $R\left(s\right)$, the boundary two-variable polynomial
matrix $\bar{\Psi}\left(\zeta,\eta\right)$ is easily computed according
to (\ref{eq:DerivativeOperatorCoeffcientMatrix}), and factorized
as in (\ref{eq:FactorBoundary2Polynomial}), as follows 
\[
\bar{\Psi}\left(\zeta,\eta\right)=\left(\begin{array}{cc}
-T\left(\zeta-\eta\right) & 0\\
0 & 0
\end{array}\right)=R_{\partial}^{\top}\left(\zeta\right)\Theta_{1}R_{\partial}\left(\eta\right)
\]
with $R_{\partial}\left(s\right)=\left(\begin{array}{cc}
1 & 0\\
-T\,s & 0
\end{array}\right)$ which admits the decomposition into $P_{\partial}(s)=\left(\begin{array}{cc}
1 & 0\end{array}\right)$ and $S_{\partial}(s)=\left(\begin{array}{cc}
-Ts & 0\end{array}\right)$ and the boundary port variable are then 
\begin{equation}
\left(\begin{array}{c}
\chi_{\partial}\\
\varepsilon_{\partial}
\end{array}\right)=\mathrm{tr}\left(\left(\begin{array}{c}
P_{\partial}(\frac{\partial}{\partial z})\\
S_{\partial}(\frac{\partial}{\partial z})
\end{array}\right)\xi\right)=\mathrm{tr}\left(\begin{array}{c}
u\\
-T\frac{\partial u}{\partial z}
\end{array}\right),\label{eq:FkexibleRodSymplectic_EnergyBPvar}
\end{equation}
which are the displacement $u$ and the stress $\sigma$ (\ref{eq:LocalElasticityRelation})
at the boundaries. 
\end{example}

\section{\label{sec:BPHS_LagrangeSibspaceDiff} Boundary port-Hamiltonian
systems defined on Stokes-Lagrange subspaces}

First, we define the generalization of boundary port-Hamiltonian
systems \cite[chap.4]{schaftGeomPhys02,Geoplex09}, and more specifically
linear boundary port-Hamiltonian systems \textit{red}{on a one-dimensional spatial domain as} defined in \cite{LeGorrecSIAM05}, by defining
the constitutive relations of the energy \emph{implicitly } using the
constitutive relation (\ref{eq:ConstRelEnergyImage}) and no longer
the linear map (\ref{eq:CoenergyExplicitLinear}). 
Second, we characterize the class of Hamiltonian functionals whose variational derivatives define the same
reciprocal operator, and within this class we identify a functional for which a balance equation in terms of the power and the energy boundary port variables holds. 
Third, we illustrate in detail this definition on the examples of the elastic rod with
local and non-local elasticity relations. Finally we conclude the section by indicating the relations between the various Hamiltonian representations presented
in this section, transforming one representation into the other.

\subsection{Definition}
In this section, we shall define a generalization of the class of boundary port-Hamiltonian
systems as presented in Section \ref{subsec:Boundary-Port-HamiltonianSyst}
by using the Stokes-Lagrange subspaces $\mathcal{L}_{\mathcal{R}}$
defined in Proposition \ref{prop:StokesLagrangeSubspace}. This definition
will first be given in geometric terms and then in terms of a system
of partial differential-algebraic equations. Then we shall relate
this system of partial differential-algebraic equations with the definition
of descriptor port-Hamiltonian systems suggested in \cite{Beattie_MMCS_2018}.

\begin{definition}
\label{def:Linear_Lagrangian_BPHS}{[} Boundary port-Hamiltonian
system on a Stokes-Lagrange subspace{]} Let $\mathcal{X}=L^{2}\left(\left[a,\,b\right],\,\mathbb{R}^{n}\right)$
be the state space, $\mathcal{R}$ a maximally reciprocal operator as in (\ref{eq:MatrixOperatorConstRelEnergy}),
$\mathcal{J}$ an  $N$th-order
Hamiltonian matrix differential operator $\mathcal{J}$ as in
(\ref{eq:N_OrderLinearHamOp}). Denote by
 \\  $\mathscr{D}_{\mathcal{J}}$
 the Stokes-Dirac structure (\ref{eq:StokesDiracConstEq}) associated with $\mathcal{J}$ and 
 \\ $\mathcal{L}_{\mathcal{R}}$ the Stokes-Lagrange subspace (\ref{prop:StokesLagrangeSubspace}) generated by $\mathcal{R}$\\
 A \emph{linear boundary port-Hamiltonian system} is defined by the dynamical system 
\begin{equation}
\left(\left(\begin{array}{c}
\frac{\partial x}{\partial t}\\
f_{\partial}
\end{array}\right),\left(\begin{array}{c}
e\\
e_{\partial}
\end{array}\right)\right)\in\mathscr{D}_{\mathcal{J}}\qquad\textrm{and}\qquad\left(\left(\begin{array}{c}
x\\
\chi_{\partial}
\end{array}\right),\left(\begin{array}{c}
e\\
\varepsilon_{\partial}
\end{array}\right)\right)\in\mathcal{L}_{\mathcal{R}}\label{eq:LagrangeBPHS_Geometric}
\end{equation}
where $e\in\mathcal{E}=L^{2}\left(\left[a,\,b\right],\,\mathbb{R}^{n}\right)$
is called the vector of \emph{co-energy variables}, $\left(f_{\partial},\,e_{\partial}\right)\in\mathbb{R}^{p}\times\mathbb{R}^{p}$
are called \emph{power boundary port variables} and $\left(\chi_{\partial},\,e_{\partial}\right)\in\mathbb{R}^{m}\times\mathbb{R}^{m}$
are called \emph{energy boundary port variables.} 
\end{definition}

\begin{note}
Note that this definition is a strict extension of Definition
\ref{def:Linear_BPHS} of a boundary port-Hamiltonian system defined
with respect to a Stokes-Dirac structure by (\ref{eq:BPHS_GeometricDef}),
where the co-energy variable $e$ is no longer a (linear) function
of the state $x$, but instead $e$ and $x$ are related by the Stokes-Lagrange
subspace $\mathcal{L}_{\mathcal{R}}$. Furthermore, the interaction
of the system with its environment at the boundary is not only by
the power-conjugate port variables (as in (\ref{eq:BPHS_GeometricDef})),
but also by a novel pair of \emph{energy-conjugate} variables $\left(\chi_{\partial},\,\varepsilon_{\partial}\right)$
associated with the Stokes-Lagrange subspace.
\end{note}

Using the image representation of the Stokes-Lagrange subspace (\ref{eq:ConstRelEnergyImage})
and the kernel representation of the Stokes-Dirac structure (\ref{eq:StokesDiracConstEq}),
the geometric definition of boundary port-Hamiltonian systems (\ref{eq:LagrangeBPHS_Geometric}),
gives rise to the following system of equations 
\begin{eqnarray}
\frac{\partial}{\partial t}\mathcal{P}\xi & = & \mathcal{J}\mathcal{S}\xi\label{eq:LagrangeBPHS_PDE}\\
\left(\begin{array}{c}
f_{\partial}\\
e_{\partial}
\end{array}\right) & = & \frac{1}{\sqrt{2}}\left(\begin{array}{cc}
\Sigma & \Sigma\\
\Sigma & -\Sigma
\end{array}\right)\mathrm{tr}\left(T\left(\frac{\partial}{\partial z}\right)e\right)\label{eq:LagrangeBPHS_SD_BoundaryPortVar}\\
\left(\begin{array}{c}
\chi_{\partial}\\
\varepsilon_{\partial}
\end{array}\right) & = & \mathrm{tr}\left(\left(\begin{array}{c}
P_{\partial}(\frac{\partial}{\partial z})\\
S_{\partial}(\frac{\partial}{\partial z})
\end{array}\right)\xi\right)\label{eq:LagrangeBPHS_Lagrange_BoundaryVariables}
\end{eqnarray}

\subsection{Hamiltonian functionals and their balance equation}

In this section we shall be interested in deriving a Hamiltonian functional which corresponds the total energy for physical systems' models 
and their balance equation for the boundary port-Hamiltonian systems defined above (with
respect to a Stokes-Dirac structure and a Stokes-Lagrange subspace).
Indeed in Definition \ref{def:Linear_Lagrangian_BPHS} of a boundary port-Hamiltonian systems on a Stokes-Lagrange subspace, 
contrary to Definition \ref{def:Linear_BPHS}, no Hamiltonian
functional is defined. We shall therefore first show how to associate Hamiltonian
functionals to Stokes-Lagrange subspaces (\ref{eq:SubspaceConstRel-1}).
Second, we identify a Hamiltonian functional for which a balance equation can be derived involving two
types of boundary port variables. Namely, those associated with the Stokes-Lagrange
subspace, and those associated with the Stokes-Dirac structure. Again
this will be illustrated on the elastic rod and nanorod examples.

\subsubsection{Case of a reciprocal operator of degree $0$}

Let us first discuss the case when the reciprocal operator $\mathcal{R}$
in (\ref{eq:MatrixOperatorConstRelEnergy}) is a matrix differential
operator of order $0$. Hence $R\left(s\right)=R_{0}=\left(\begin{array}{c}
P_{0}\\
S_{0}
\end{array}\right)$ is simply a real-valued $2n \times n$ matrix. In
this case the Hamiltonian functional is defined, in a way that is completely
similar to the finite-dimensional case \cite{SCL_18,Beattie_MMCS_2018},
as 
\begin{equation}
\mathfrak{H}\left(\xi\right)=\frac{1}{2}\int_{a}^{b}\xi^{\top}\left(\frac{R_{0}^{\top}\Xi_{n}R_{0}}{2}\right)\xi\,dz\label{eq:Hamiltonian_MatrixReciprocalMatrix}
\end{equation}
where 
\begin{equation}
\Xi_{n}=\left(\begin{array}{cc}
0_{n} & I_{n}\\
I_{n} & 0_{n}
\end{array}\right)\label{eq:KreinProduct}
\end{equation}

\begin{note}
Note that according to the reciprocity condition (\ref{eq:ReciprocityCondOperator0})
the Hamiltonian density function is defined by the following\emph{
symmetric} matrix 
\[
\frac{1}{2}\left(R_{0}^{\top}\Xi_{n}R_{0}\right)=\frac{1}{2}\left(P_{0}^{\top}S_{0}+S_{0}^{\top}P_{0}\right)=S_{0}^{\top}P_{0}=P_{0}^{\top}S_{0}
\]
Furthermore in the case when $P_{0}=I_{n}$ , one recovers precisely
the definition of the energy (\ref{eq:EnergyFunctionQuadratic}) with
$S_{0}=Q_{0}$, for boundary port-Hamiltonian systems defined on Stokes-Dirac
structures.
\end{note}

Hence the variational derivative of the Hamiltonian is $\frac{\delta\mathfrak{H}}{\delta\xi}\left(\xi\right)=S_{0}^{\top}P_{0}\xi$,
and one derives the time-derivative of the Hamiltonian 
\[
\frac{d\mathfrak{H}}{dt}=\int_{a}^{b}\frac{\delta\mathfrak{H}}{\delta\xi}\left(\xi\right)\frac{\partial\xi}{\partial t}\,dz=\int_{a}^{b}\left(P_{0}\xi\right)^{\top}S_{0}\frac{\partial\xi}{\partial t}\,dz=\int_{a}^{b}e^{\top}\frac{\partial x}{\partial t}\,dz
\]
The isotropy condition (\ref{eq:IsotropyCond_StokesDirac}) of the
Stokes-Dirac structure, yields the \emph{energy balance equation}
\[
\frac{d}{dt}\mathfrak{H}\left(\xi\right)=\int_{a}^{b}e^{\top}\frac{\partial x}{\partial t}\,dz=e_{\partial}^{\top}\,f_{\partial}
\]
Since the Stokes-Lagrange subspace is defined by an operator $\mathcal{R}$
of degree $0$, there are no energy port boundary variables.
One recovers the balance equation obtained for the class of boundary
port variables presented in \cite{LeGorrecSIAM05} and illustrated
by the example presented in Section \ref{subsec:Boundary-PH-formulation-Rod}.

\subsubsection{Case of a reciprocal operator of differential degree $\protect\geq1$}

Let us now consider the general case of maximally reciprocal matrix
differential operators (\ref{eq:MatrixOperatorConstRelEnergy}). The
first question is how to characterize the Hamiltonian functional associated
with the reciprocal operator (\ref{eq:MatrixOperatorConstRelEnergy})
and defining the constitutive relations (\ref{eq:ConstRelEnergyImage})
between the energy and co-energy variables. A natural extension of the
Hamiltonian (\ref{eq:Hamiltonian_MatrixReciprocalMatrix}) is 
\begin{equation}
\mathfrak{H}_{0}\left(\xi\right)=\frac{1}{2}\int_{a}^{b}\xi^{\top}\frac{R^{\top}\left(-\frac{\partial}{\partial z}\right)\Xi_{n}R\left(\frac{\partial}{\partial z}\right)}{2}\xi\,dz\label{eq:NaturalHamiltonian}
\end{equation}

Its variational derivative may be expressed as follows, using Lemma \ref{lem:DiffQuadFunctPolynom} and the expression (\ref{eq:VariationalDer_Polynomial}) of the associated two-variable polynomial matrix, 
\begin{equation}
H\left(-s,s\right)=\frac{1}{2}R^{\top}\left(-s\right)\Xi_{n}R\left(s\right)=S^{\top}\left(-s\right)P\left(s\right)\label{eq:2Polynomial_NonLocalConstRelMatrices}
\end{equation}
where the expression involving the polynomial matrices $P\left(s\right)$ and $S\left(s\right)$ is obtained by using the reciprocity condition (\ref{eq:ReciprocityConditionsPolynomialMatrix}).

However, Hamiltonian functionals 
\begin{equation}
\mathfrak{H}\left(\xi\right)=\frac{1}{2}\int_{a}^{b}\sum_{i,j=1}^{M}\frac{\partial^{i}\xi}{\partial z^{i}}^{\top}\,H_{ij}\frac{\partial^{j}\xi}{\partial z^{j}}\,dz\label{eq:QuadraticHamiltonian}
\end{equation}
defined by a \emph{symmetric} two-variable polynomial matrix
$H\left(\varsigma,\,\eta\right)$, satisfying (\ref{eq:2Polynomial_NonLocalConstRelMatrices}),
are not uniquely defined.
In fact, the general solution $H\left(\varsigma,\,\eta\right)$
of equation (\ref{eq:2Polynomial_NonLocalConstRelMatrices}) is 
\begin{equation}
H\left(\varsigma,\,\eta\right)=\frac{1}{2}R^{\top}\left(\varsigma\right)\Xi_{n}R\left(\eta\right)+\left(\varsigma+\eta\right)\Gamma\left(\varsigma,\,\eta\right)\label{eq:2Polynomials_general_ExplicitRecOp}
\end{equation}
where $\Gamma\left(\varsigma,\,\eta\right)$ is an arbitrary symmetric
two-variable polynomial matrix of degree $\left(M-1\right)$. 

In the following proposition, we identify the Hamiltonian
functions that satisfies a balance equation in terms of the power and
energy boundary port variables. 
\begin{proposition}
Consider a linear boundary port-Hamiltonian system as in Definition
\ref{def:Linear_Lagrangian_BPHS} and define the Hamiltonian \textup{$\mathfrak{H}\left(\xi\right)$}
associated with the two-variable polynomial matrix as
\begin{equation}
H\left(\varsigma,\,\eta\right)=\frac{R^{\top}\left(\varsigma\right)\Xi_{n}R\left(\eta\right)}{2}-\left(\varsigma+\eta\right)\frac{R_{\partial}^{\top}\left(\varsigma\right)\Xi_{p}R_{\partial}\left(\eta\right)}{2}\label{eq:CanonicalHamiltonian}
\end{equation}
where $R\left(s\right)$ is the polynomial matrix associated with
the reciprocal operator defining the Lagrangian subspace and $R_{\partial}\left(s\right)$
is its boundary operator (\ref{eq:BoundaryMapRecip_Decomposed}).
This Hamiltonian satisfies the balance equation 
\begin{eqnarray}
\frac{d\mathfrak{H}}{dt} & = & e_{\partial}^{\top}\,f_{\partial}-\left[\varepsilon_{\partial}^{\top}\frac{d\chi_{\partial}}{dt}\right]_{a}^{b}\label{eq:BalanceHamiltonian}
\end{eqnarray}
\end{proposition}

\begin{proof}
First recall that $R^{\top}\left(\varsigma\right)\Xi_{n}R\left(\eta\right)=P^{\top}\left(\varsigma\right)S\left(\eta\right)+S^{\top}\left(\varsigma\right)P\left(\eta\right)$
and that $R_{\partial}^{\top}\left(\varsigma\right)\Xi_{p}R_{\partial}\left(\eta\right)=P_{\partial}^{\top}\left(\varsigma\right)S_{\partial}\left(\eta\right)+S_{\partial}^{\top}\left(\varsigma\right)P_{\partial}\left(\eta\right)$.
The derivative of the Hamiltonian associated with the two-variable polynomial matrix (\ref{eq:CanonicalHamiltonian})
is then 
\begin{eqnarray*}
\frac{d\mathfrak{H}}{dt} & = & \frac{1}{2}\int_{a}^{b}\left(P\left(\frac{\partial}{\partial z}\right)\xi\right)^{\top}S\left(\frac{\partial}{\partial z}\right)\frac{\partial\xi}{\partial t}\,dz+\frac{1}{2}\int_{a}^{b}\left(S\left(\frac{\partial}{\partial z}\right)\xi\right)^{\top}P\left(\frac{\partial}{\partial z}\right)\frac{\partial\xi}{\partial t}\,dz\\
 &  & -\frac{1}{2}\left[\left(P_{\partial}\left(\frac{\partial}{\partial z}\right)\xi\right)^{\top}S_{\partial}\left(\frac{\partial}{\partial z}\frac{\partial\xi}{\partial t}\right)+\left(S_{\partial}\left(\frac{\partial}{\partial z}\right)\xi\right)^{\top}P_{\partial}\left(\frac{\partial}{\partial z}\right)\frac{\partial\xi}{\partial t}\right]_{a}^{b}
\end{eqnarray*}
Applying (\ref{eq:IntByParts_BoundaryEnergyVar}) with $\xi_1=\xi$ and
$\xi_2=\frac{\partial\xi}{\partial t}$, one has
\begin{eqnarray*}
\int_{a}^{b}\left(P\left(\frac{\partial}{\partial z}\right)\xi\right)^{\top}S\left(\frac{\partial}{\partial z}\right)\frac{\partial\xi}{\partial t}\,dz = \int_{a}^{b}\left(S\left(\frac{\partial}{\partial z}\right)\xi\right)^{\top}P\left(\frac{\partial}{\partial z}\right)\frac{\partial\xi}{\partial t}\,dz \\
 -\left[\left(S_{\partial}\left(\frac{\partial}{\partial z}\right)\xi\right)^{\top}P_{\partial}\left(\frac{\partial}{\partial z}\right)\frac{\partial\xi}{\partial t}-\left(S_{\partial}\left(\frac{\partial}{\partial z}\right)\frac{\partial\xi}{\partial t}\right)^{\top}P_{\partial}\left(\frac{\partial}{\partial z}\right)\xi\right]_{a}^{b}
\end{eqnarray*}
Hence, one obtains 
\begin{eqnarray*}
\frac{d\mathfrak{H}}{dt} & = & \int_{a}^{b}\left(S\left(\frac{\partial}{\partial z}\right)\xi\right)^{\top}P\left(\frac{\partial}{\partial z}\right)\frac{\partial\xi}{\partial t}\,dz\\
 &  & -\frac{1}{2}\left[\left(S_{\partial}\left(\frac{\partial}{\partial z}\right)\xi\right)^{\top}P_{\partial}\left(\frac{\partial}{\partial z}\right)\frac{\partial\xi}{\partial t}-\left(S_{\partial}\left(\frac{\partial}{\partial z}\right)\frac{\partial\xi}{\partial t}\right)^{\top}P_{\partial}\left(\frac{\partial}{\partial z}\right)\xi\right]_{a}^{b}\\
 &  & -\frac{1}{2}\left[\left(S_{\partial}\left(\frac{\partial}{\partial z}\right)\xi\right)^{\top}P_{\partial}\left(\frac{\partial}{\partial z}\right)\frac{\partial\xi}{\partial t}+\left(S_{\partial}\left(\frac{\partial}{\partial z}\right)\frac{\partial\xi}{\partial t}\right)^{\top}P_{\partial}\left(\frac{\partial}{\partial z}\right)\xi\right]_{a}^{b}\\
 & = & \int_{a}^{b}\left(S\left(\frac{\partial}{\partial z}\right)\xi\right)^{\top}P\left(\frac{\partial}{\partial z}\right)\frac{\partial\xi}{\partial t}\,dz-\left[\left(S_{\partial}\left(\frac{\partial}{\partial z}\right)\xi\right)^{\top}P_{\partial}\left(\frac{\partial}{\partial z}\right)\frac{\partial\xi}{\partial t}\right]_{a}^{b}
\end{eqnarray*}
Permuting the derivatives with respect to time and space, using the
constitutive relations of the energy and co-energy variables (\ref{eq:ConstRelEnergyImage}),
and the isotropy condition (\ref{eq:IsotropyCond_StokesDirac_expanded})
of the Stokes-Dirac structure, one obtains the balance equation 

\begin{eqnarray*}
\frac{d\mathfrak{H}}{dt}  & = & \int_{a}^{b}\left(S\left(\frac{\partial}{\partial z}\right)\xi\right)^{\top}\,\frac{\partial}{\partial t}P\left(\frac{\partial}{\partial z}\right)\xi\,dz-\left[\left(S_{\partial}\left(\frac{\partial}{\partial z}\right)\xi\right)^{\top}\,\frac{\partial}{\partial t}P_{\partial}\left(\frac{\partial}{\partial z}\right)\xi\right]_{a}^{b}\\
 & = & \int_{a}^{b}e^{\top}\frac{\partial x}{\partial t}\,dz-\left[\varepsilon_{\partial}^{\top}\frac{d\chi_{\partial}}{dt}\right]_{a}^{b}\\
 & = & e_{\partial}^{\top}\,f_{\partial}-\left[\varepsilon_{\partial}^{\top}\frac{d\chi_{\partial}}{dt}\right]_{a}^{b}
\end{eqnarray*}

\end{proof}

The particular choice of the Hamiltonian (\ref{eq:CanonicalHamiltonian}) is primarily motivated by the balance equation (\ref{eq:BalanceHamiltonian}) in terms of the power boundary variables of the Stokes-Dirac structure and the energy port variables of the Stokes-Lagrange subspace. Indeed if one would choose a Hamiltonian (\ref{eq:2Polynomials_general_ExplicitRecOp}) with an arbitrary symmetric two-variable polynomial matrix $\Gamma\left(\varsigma,\,\eta\right)$, the boundary terms in its balance equation are no more expressed in terms of the boundary port variables. The multiplicity of possible boundary terms (modulo integration by parts) in the energy balance equations is well-known and has been discussed in the context of second-order field theories \cite{Schoeberl_JMathPhysics2018}. However, there might be other choices of the Hamiltonian functional where the boundary terms of its balance equation depend on the boundary port variables. For instance, choose  $\Gamma\left(\varsigma,\,\eta\right)=0$ or equivalently the Hamiltonian functional (\ref{eq:Hamiltonian_MatrixReciprocalMatrix}). Then, using the same calculations as in the proof, one obtains the following balance equation
\begin{equation}
\frac{d\mathfrak{H}_{0}}{dt} =e_{\partial}^{\top}\,f_{\partial} -  \frac{1}{2} \left[\varepsilon_{\partial}^{\top}\frac{d\chi_{\partial}}{dt} - \frac{d\varepsilon_{\partial}}{dt}^{\top}\chi_{\partial} \right]_{a}^{b} \label{eq:BalanceHamiltonian_MatrixReciprocalMatrix}
\end{equation}
This expression is more complex than the balance equation, mixing the time derivative of both energy port variables. It appears also that for physical systems' models, the expression of the Hamiltonian $\mathfrak{H}_{0}$ (\ref{eq:Hamiltonian_MatrixReciprocalMatrix}) does not correspond to the usual physical expression but rather the Hamiltonian (\ref{eq:CanonicalHamiltonian}). We shall illustrate this in Remark \ref{note:DiscussionHamiltoniansElasticRod} on the example of the elastic rod.

\subsection{Examples}

As a first example, consider again the symplectic formulation of the
elastic rod with local elasticity relations. 
\begin{example}
\label{exa:ElasticRodSympl_LagrangeBPHS} Consider the symplectic formulation
of the elastic rod presented in Section \ref{subsec:Symplectic-formulation-of-rod}.
In this case the Hamiltonian operator is the symplectic matrix. Hence
it contains no derivations and the associated Dirac structure is just
the graph of a skew-symmetric matrix and there are no power boundary
variables associated with it. The state is $x^{\top}=\left(\begin{array}{cc}
u & p\end{array}\right)$, where $u$ is the displacement of a point of the rod and $p$ the
momentum. Furthermore, the reciprocal operator $\mathcal{R}$ defining
the constitutive relations in (\ref{eq:ConstRelEnergyImage}), is
the differential operator defined by the operators $\mathcal{P}=P_{0}=I_{2}$
and the differential operator $\mathcal{S}=\textrm{diag\ensuremath{\left(k-T\left(\frac{\partial}{\partial z}\right)^{2},\,\frac{1}{\rho A}\right)}}$. The boundary port-Hamiltonian system is then given by the Hamiltonian system (\ref{eq:VibrRod_Sympl_DiffQ}) augmented with the \emph{energy}
boundary variables derived in Example \ref{exa:Rod_EnergyBoundaryVariablesDiffQ}
\[
\left(\begin{array}{c}
\chi_{\partial}\\
\varepsilon_{\partial}
\end{array}\right)=\mathrm{tr}\left(\begin{array}{c}
u\left(z\right)\\
T\frac{\partial u}{\partial z}\left(z\right)
\end{array}\right)
\]
being the displacement $u$ and the stress at the boundaries.

Let us now define the Hamiltonian functional associated with the reciprocal
operator $\mathcal{R}$ and derive its balance equation. Calculating
the two-variable polynomial matrix(\ref{eq:CanonicalHamiltonian}),
one obtains 
\begin{eqnarray*}
H\left(\varsigma,\,\eta\right) & = & \left(\begin{array}{cc}
k-\frac{1}{2}T\left(\varsigma^{2}+\eta^{2}\right) & 0\\
0 & \frac{1}{\rho A}
\end{array}\right)-\frac{1}{2}\left(\varsigma+\eta\right)\left(\begin{array}{cc}
-T\left(\varsigma+\eta\right) & 0\\
0 & 0
\end{array}\right)\\
 & = & \left(\begin{array}{cc}
k+\varsigma\eta & 0\\
0 & \frac{1}{\rho A}
\end{array}\right)
\end{eqnarray*}
and, by (\ref{eq:QuadraticHamiltonian}), this defines the Hamiltonian
\begin{equation}
\mathfrak{H}\left(u\right)=\frac{1}{2}\int_{a}^{b}\left(k\,u^{2}+T\left(\frac{\partial u}{\partial z}\right)^{2}+\frac{p^{2}}{\rho A}\right)\,dz
\label{eq:TotalEnergyElasticRodLocal}
\end{equation}
which is precisely the mechanical energy consisting of the sum of
the elastic energy, the structural elastic energy, and the kinetic
energy defined in the physical model in Example \ref{subsec:Symplectic-formulation-of-rod}.

The energy balance equation, derived from (\ref{eq:BalanceHamiltonian})
and the energy boundary variables (\ref{eq:FkexibleRodSymplectic_EnergyBPvar}),
is 
\begin{equation}
\frac{d\mathfrak{H}}{dt}=-\left[-\left(T\frac{\partial u}{\partial z}\right)\,\frac{du}{dt}\right]_{a}^{b}=\left[\left(T\frac{\partial u}{\partial z}\right)\,\frac{du}{dt}\right]_{a}^{b}
\label{eq:BalanceEqFlexRod}
\end{equation}
where the right-hand side is the mechanical power at the boundary
of the rod: the product of the strain with the velocity. 
\end{example}

\begin{note} \label{note:DiscussionHamiltoniansElasticRod}
In this example, we shall illustrate the alternative choice of Hamiltonian function given in (\ref{eq:Hamiltonian_MatrixReciprocalMatrix}) which is in the example of the elastic rod with local elasticity relation
\begin{equation}
\mathfrak{H}_0\left(u\right)=\frac{1}{2}\int_{a}^{b}\left(k\,u^{2}-T\,u\left(\frac{\partial^{2} u}{\partial z^{2}}\right)+\frac{p^{2}}{\rho A}\right)\,dz
\label{eq:H0ElasticRod}
\end{equation}
associated with the two-variable polynomial matrix
\[
H_0\left(\varsigma,\,\eta\right)  =  \left(\begin{array}{cc}
k-\frac{1}{2}T\left(\varsigma^{2}+\eta^{2}\right) & 0\\
0 & \frac{1}{\rho A}
\end{array}\right)
\]
The expression (\ref{eq:H0ElasticRod}) is indeed not the classical expression of the total energy of the rod but rather (\ref{eq:TotalEnergyElasticRodLocal}). Furthermore, the balance equation of the Hamiltonian functional becomes
\[
\frac{d\mathfrak{H}_0}{dt}=\frac{1}{2}\left[\left(T\frac{\partial u}{\partial z}\right)\,\frac{du}{dt} - \, \frac{d\left(T\frac{\partial u}{\partial z}\right)}{dt}  u \right]_{a}^{b}
\]
which has a far less easy interpretation than the balance equation (\ref{eq:BalanceEqFlexRod}).
\end{note}

As a second example, consider the elastic rod with non-local elasticity
relations, where both the Hamiltonian operator and the constitutive
relations of the co-energy variables induce boundary port variables. 
\begin{example}
\label{exa:NonlocalElastic_LagrangeBPHS} Consider again the example,
presented in Section \ref{subsec:Descriptor-formulation-NonlocalRod},
of an elastic rod made of composite material with non-local elasticity
relations (\ref{eq:NonlocalElasticityRelation}). Recall that the
elasticity relations are expressed as defining the strain $\epsilon$
and the stress $\sigma$ as function of the latent strain $\lambda$
in (\ref{eq:ImageReprNonLocalElasticity}). Using this expression,
the energy variables $x=\left(\begin{array}{ccc}
u & \epsilon & p\end{array}\right)^{\top}$ and the co-energy variables $e=\left(\begin{array}{ccc}
F^{el} & \sigma & v\end{array}\right)^{\top}$ are a function (\ref{eq:ConstRelEnergyImage}) of the latent variables
$\xi=\left(\begin{array}{ccc}
u & \lambda & p\end{array}\right)^{\top}$ defined by the reciprocal operator $\mathcal{R}$ in (\ref{eq:ConstRelEnergyImage})
associated with the operators $\mathcal{P}=\textrm{diag\ensuremath{\left(1,1-\mu\left(\frac{\partial}{\partial z}\right)^{2},\,1\right)}}$
and $\mathcal{S}=\textrm{diag\ensuremath{\left(k,T,\frac{1}{\rho A}\right)}}$.

From the polynomial matrix (\ref{eq:2VarPolynomialMatrix_R}) associated
with 
\[
R\left(s\right)=\left(\begin{array}{c}
P\left(s\right)\\
S\left(s\right)
\end{array}\right)=\left(\begin{array}{c}
\textrm{diag}\left(1,1-\mu s^{2},\,1\right)\\
\textrm{diag}\ensuremath{\left(k,T,\frac{1}{\rho A}\right)}
\end{array}\right)
\]
the boundary two-variable polynomial matrix $\bar{\Psi}\left(\zeta,\eta\right)$
is computed and factorized as in (\ref{eq:FactorBoundary2Polynomial}),
as follows 
\[
\bar{\Psi}\left(\zeta,\eta\right)=\textrm{diag}\left(0,T\,\mu\,\left(\zeta-\eta\right),\,0\right)=R_{\partial}^{\top}\left(\zeta\right)\Theta_{1}R_{\partial}\left(\eta\right)
\]
with

\begin{equation}
R_{\partial}\left(s\right)=\left(\begin{array}{ccc}
0 & T & 0\\
0 & \mu s & 0
\end{array}\right)\label{eq:BoundaryPolynomialNonlocalElasticity}
\end{equation}
Hence the boundary port-Hamiltonian model of the rod with non-local
elasticity relations is defined with respect to the Stokes-Dirac structure
$\mathscr{D}_{\mathcal{J}_{1}}$ associated with the Hamiltonian operator
$\mathcal{J}_{1}$ defined in (\ref{eq:HamSyst_VibrRod_Symmetry})
and the Stokes-Lagrange subspace $\mathcal{L}_{\mathcal{R}}$ associated
with the reciprocal operator $\mathcal{R}$ defined above. This leads
to the following equations defined by the Hamiltonian system (\ref{eq:HamSysNonlocalElasticityRelation})
augmented with the \emph{power} boundary variables (\ref{eq:BPvariables_VibrRod_Symmetry})
of the Stokes-Dirac structure $\mathscr{D}_{\mathcal{J}_{1}}$ 
\[
\left(\begin{array}{c}
f_{\partial}\\
e_{\partial}
\end{array}\right)=\left(\begin{array}{ccc}
0 & 0 & 1\\
0 & 1 & 0
\end{array}\right)\mathrm{tr}\,Q_{0}\left(\begin{array}{c}
u\\
\lambda\\
p
\end{array}\right)=\mathrm{tr}\left(\begin{array}{c}
\frac{p}{\left(\rho A\right)}\\
T\;\lambda
\end{array}\right)
\]
being the velocity $\frac{p}{\left(\rho A\right)}$ and, according
to the image representation (\ref{eq:ImageReprNonLocalElasticity}),
the stress $\sigma=T\;\lambda$ at the boundaries. Using (\ref{eq:BoundaryPolynomialNonlocalElasticity}),
the \emph{energy} boundary variables (\ref{eq:EnergyBoundaryPortVariables})
of the Stokes-Lagrange subspace $\mathcal{L}_{\mathcal{R}}$ are computed
as 
\begin{equation}
\left(\begin{array}{c}
\chi_{\partial}\\
\varepsilon_{\partial}
\end{array}\right)=\mathrm{tr}\left(\begin{array}{c}
T\lambda\\
\mu\,\frac{\partial\lambda}{\partial z}
\end{array}\right)\label{eq:NonlocalElastic_EnergyBoundaryVar}
\end{equation}
being the stress $\sigma=T\;\lambda$ and the nonlocal strain
component $\mu\,\frac{\partial\lambda}{\partial z}$ , evaluated at
the boundaries.

Let us now derive the Hamiltonian functional and its balance equation
from the reciprocal operator defining the Stokes-Lagrange subspace.
With a very similar calculation as in the preceding example, one obtains
the two-variable polynomial matrix 
\[
H\left(\varsigma,\,\eta\right)=\textrm{diag}\left(\begin{array}{ccc}
k & \left[T+T\mu\,\varsigma\eta\right] & \frac{1}{\rho A}\end{array}\right)
\]
defining the Hamiltonian 
\[
\mathfrak{H}\left(u\right)=\frac{1}{2}\int_{a}^{b}\sum_{i,j=1}^{M}\left(k\,u^{2}+T\lambda^{2}+\mu T\left(\frac{\partial\lambda}{\partial z}\right)^{2}+\frac{p^{2}}{\rho A}\right)\,dz
\]
which may again be identified with the mechanical energy. This time,
the Hamiltonian operator is a differential operator and for the balance
equation the boundary port variables of the Stokes-Dirac structure
need to be taken into account. Recalling the expression of port boundary
variables of the Stokes-Dirac structure (\ref{eq:BPvariables_VibrRod_Symmetry})
associated with the Hamiltonian operator $\mathcal{J}_{1}$ in \ref{eq:HamSysNonlocalElasticityRelation},
one obtains 
\[
\left(\begin{array}{c}
f_{\partial}\\
e_{\partial}
\end{array}\right)=\left(\begin{array}{ccc}
0 & 0 & 1\\
0 & 1 & 0
\end{array}\right)\mathrm{tr}\,Q_{0}\left(\begin{array}{c}
u\\
\lambda\\
p
\end{array}\right)=\mathrm{tr}\left(\begin{array}{c}
\frac{p}{\left(\rho A\left(z\right)\right)}\\
T\left(z\right)\;\lambda\left(z\right)
\end{array}\right)
\]
The energy boundary variables are given in (\ref{eq:NonlocalElastic_EnergyBoundaryVar}).
Hence the balance equation (\ref{eq:BalanceHamiltonian}) becomes
\begin{equation}
\frac{d\mathfrak{H}}{dt}=\left[T\lambda\frac{p}{\left(\rho A\left(z\right)\right)}\right]_{a}^{b}-\left[\mu\,\frac{\partial\lambda}{\partial z}\,\frac{d\left(T\lambda\right)}{dt}\right]_{a}^{b}\label{eq:NonLocalRod_Energybalance}
\end{equation}
where the right-hand side is the sum of the mechanical power at the boundary of the rod, due to local elasticity (depending on the coefficient $T$ only) and nonlocal elasticity (depending on the coefficient $\mu$). 

Notice that when $\mu=0$, that is the elasticity relation are only
defined by the local elasticity relation (\ref{eq:LocalElasticityRelation}),
then the local and the non-local strain are equal, $\epsilon=\lambda$,
and this Hamiltonian formulation reduces to the port-Hamiltonian
system (\ref{eq:LocalElasticityRelation}) with respect to the Hamiltonian
functional (\ref{eq:Hamiltonian0_VibratingRod}). Furthermore, in the energy
balance equation (\ref{eq:NonLocalRod_Energybalance}) only the term
due to the power port variables remains.

\end{example}

\section{Conclusions}

In this paper, we have extended the definition of boundary port-Hamiltonian
systems to the case where the Hamiltonian functional is not explicitly
given, but instead a Lagrangian subspace is defined in the cotangent
bundle of the state variables. This \emph{Lagrangian subspace is defined
by maximally reciprocal differential operators}, where reciprocal
refers to Maxwell's reciprocity relations, extending the case of the
graph of the differential of a Hamiltonian functional. For physical
systems, these reciprocal operators correspond to the constitutive relations
associated with the energy of the system; for instance the elasticity
constitutive relations of elastic rods. It has then be shown that
for open systems one has to introduce port variables, called \emph{energy
boundary port variables}, associated with the reciprocal differential
operators, and their expression in terms of boundary operators has been
derived. Finally, we have defined boundary port-Hamiltonian systems
defined with respect to a Stokes-Dirac structure and a Stokes-Lagrange
subspace. This strictly generalizes the previous definition of boundary
port-Hamiltonian systems. The main extension is that it includes the symplectic formulations of Hamiltonian systems and allows to handle
non-local constitutive relations. Furthermore, an important difference
is that it has now two types of port variables: first, the power
port boundary variables associated with the Stokes-Dirac structure,
and second, the energy port boundary variables associated with the
Stokes-Lagrange subspace. It should be observed that the definition
of the energy port boundary variables constitutes a strict generalization
of the descriptor port-Hamiltonian systems suggested in \cite{Beattie_MMCS_2018,Zwart_MCMDS19_nonlocal}.
Finally we have identified a Hamiltonian functional associated with
the Lagrangian subspace and derived its balance equation which contains two types of terms: the product of the power boundary
variables minus the product of the (time derivative of the) energy boundary port variables.

All this was illustrated by the running example of the elastic rod. However, many other physical systems are of the same type. For example, the same equations appear in the Hamiltonian formulation of the dynamics of a transmission line. See in particular \cite{jeltsema} for the symplectic formulation of the telegrapher's equations. In future work we also plan to extend the obtained results to systems on higher-dimensional spatial domains.

Although the present paper has only considered \emph{linear} boundary control port-Hamiltonian systems, many
of the results can be extended to the case of \emph{nonlinear} reciprocal
operators; of course with a more general definition of the boundary
operators defining the energy boundary port variables. Another continuation
of this work is to use the theory of semi-group operator theory and descriptor
systems in order to derive the properties of this extended class of boundary
port-Hamiltonian systems.


\bibliographystyle{siamplain}

\begin{thebibliography}{10}

\bibitem{Abraham_marsden87}
R.~Abraham and J.~E. Marsden.
\newblock {\em Foundations of {M}echanics}.
\newblock Benjamin Cummings Publ. Comp., Reading, MA, U.S.A., ii edition, 1987.
\newblock ISBN 0-8053-0102-X.

\bibitem{Achenbach2003reciprocity}
J.D.~Achenbach.
\newblock {\em Reciprocity in elastodynamics}.
\newblock Cambridge University Press, 2003.

\bibitem{Arnold89}
V.I.~Arnold.
\newblock {\em Mathematical Methods of Classical Mechanics}.
\newblock Springer, ii edition, 1989.
\newblock ISBN 0-387-96890-3.

\bibitem{Augner_SCL20_ExpStabBPHS}
B.~Augner and H.~Laasri.
\newblock Exponential stability for infinite-dimensional non-autonomous
  port-h{H}amiltonian systems.
\newblock {\em Systems \& Control Letters}, 144:104757, 2020.

\bibitem{Baaiu_MCMDS_09}
A.~Baaiu, F.~Couenne, D.~Eberard, C.~Jallut, Y.~Le~Gorrec, L.~Lef{\`e}vre, and
  M.~Tayakout-Fayolle.
\newblock Port-based modelling of mass transfer phenomena.
\newblock {\em Mathematical and Computer Modelling of Dynamical Systems},
  15(3):233--254, 2009.

\bibitem{Beattie_MMCS_2018}
C.~Beattie, V.~Mehrmann, H.~Xu, and H.~Zwart.
\newblock Linear port-{H}amiltonian descriptor systems.
\newblock {\em Mathematics of Control, Signals, and Systems}, 30(4):17, Oct
  2018.

\bibitem{Brugnoli_AMM2019_I}
A.~Brugnoli, D.~Alazard, V.~Pommier-Budinger, and D.~Matignon.
\newblock Port-{H}amiltonian formulation and symplectic discretization of plate
  models part i: Mindlin model for thick plates.
\newblock {\em Applied Mathematical Modelling}, 75:940 -- 960, 2019.

\bibitem{Brugnoli_AMM2019_II}
A.~Brugnoli, D.~Alazard, V.~Pommier-Budinger, and D.~Matignon.
\newblock Port-{H}amiltonian formulation and symplectic discretization of plate
  models part ii: Kirchhoff model for thin plates.
\newblock {\em Applied Mathematical Modelling}, 75:961 -- 981, 2019.

\bibitem{Callen85}
H.B. Callen.
\newblock {\em Introduction to Thermodynamics and Thermostatistics, 1985}.
\newblock John Wiley and Sons, second edition.

\bibitem{Eringen_JAP1983}
A~Cemal~Eringen.
\newblock On differential equations of nonlocal elasticity and solutions of
  screw dislocation and surface waves.
\newblock {\em Journal of Applied Physics}, 54(9):4703--4710, 1983.

\bibitem{Courant88}
T.J. Courant and A.~Weinstein.
\newblock Beyond {P}oisson structures.
\newblock In {\em S{\'e}minaire {S}ud-{R}hodanien de {G}{\'e}om{\'e}trie},
  volume~8 of {\em Travaux en cours}, Paris, 1988. Hermann.

\bibitem{Geoplex09}
V.~Duindam, A.~Macchelli, S.~Stramigioli, and H.~eds. Bruyninckx.
\newblock {\em Modeling and Control of Complex Physical Systems - The
  {P}ort-{H}amiltonian Approach}.
\newblock Springer, Sept. 2009.
\newblock ISBN 978-3-642-03195-3.

\bibitem{Hamroun10}
H.~Hamroun, A.~Dimofte, L.~Lef\`evre, and E.~Mendes.
\newblock Control by interconnection and energy shaping methods of port
  {H}amiltonian models - application to the shallow water equations.
\newblock {\em European Journal of Control}, 16(5):545--563, 2010.

\bibitem{Zwart_MCMDS19_nonlocal}
H.~Heidari and H.~Zwart.
\newblock Port-{H}amiltonian modelling of nonlocal longitudinal vibrations in a
  viscoelastic nanorod.
\newblock {\em Mathematical and Computer Modelling of Dynamical Systems},
  0(0):1--16, 2019.

\bibitem{Jacob_IEEE_TAC19}
B.~{Jacob} and J.~T. {Kaiser}.
\newblock On exact controllability of infinite-dimensional linear
  port-{H}amiltonian systems.
\newblock {\em IEEE Control Systems Letters}, 3(3):661--666, 2019.

\bibitem{JacobZwart12}
B.~Jacob and H.J. Zwart.
\newblock {\em Linear {P}ort-{H}amiltonian {S}ystems on {I}nfinite-dimensional
  {S}paces}, volume 223 of {\em Operator Theory: Advances and Applications}.
\newblock Springer Basel, 2012.

\bibitem{jeltsema}
D. Jeltsema and A.J. van der Schaft.
\newblock Pseudo-gradient and {L}agrangian boundary control system
formulation of electromagnetic fields.
\newblock {\it J. Phys. A: Math. Theor.}, 40, 11627-11643, 2007.


\bibitem{Karlicic_EJM2015}
D.~Karli\v{c}i{\`c}, Milan Caji{\`c}, T.~Murmu, and S.~Adhikari.
\newblock Nonlocal longitudinal vibration of viscoelastic coupled
  double-nanorod systems.
\newblock {\em European Journal of Mechanics - A/Solids}, 49:183 -- 196, 2015.

\bibitem{LeGorrecSIAM05}
Y.~Le~Gorrec, H.~Zwart, and B.M. Maschke.
\newblock Dirac structures and boundary control systems associated with
  skew-symmetric differential operators.
\newblock {\em SIAM J. of Control and Optimization}, 44(5):1864--1892, 2005.

\bibitem{Libermann_marle87}
P.~Libermann and C.-M.~Marle.
\newblock {\em Symplectic {G}eometry and {A}nalytical {M}echanics}.
\newblock D. Reidel Publishing Company, Dordrecht, Holland, 1987.
\newblock ISBN 90-277-2438-5.

\bibitem{Macchelli_SIAM14}
A.~Macchelli.
\newblock Passivity-based control of implicit port-{H}amiltonian systems.
\newblock {\em SIAM Journal on Control and Optimization}, 52(4):2422--2448,
  2014.

\bibitem{Macchelli_2021OverviewControlBPHS}
A~Macchelli, Y~Le Gorrec, H~Ram{\'\i}rez, H~Zwart, and F~Califano.
\newblock Control design for linear port-{H}amiltonian boundary control systems:
  An overview.
\newblock {\em Stabilization of {D}istributed {P}arameter {S}ystems: {D}esign {M}ethods
  and {A}pplications}, pages 57--72, 2021.

\bibitem{MacchelliSIAM04}
A.~Macchelli and C.~Melchiorri.
\newblock Modeling and control of the {T}imoshenko beam. the {D}istributed
  {P}ort {H}amiltonian approach.
\newblock {\em SIAM Journal On Control and Optimization}, 43(2):743--767, 2004.

\bibitem{Macchelli_Automatica18_DissContrBPHS}
A.~Macchelli and Federico Califano.
\newblock Dissipativity-based boundary control of linear distributed
  port-hamiltonian systems.
\newblock {\em Automatica}, 95:54 -- 62, 2018.

\bibitem{maschkeFAP04}
B.~Maschke and A.J. van~der Schaft.
\newblock {\em Advanced Topics in Control Systems Theory. Lecture Notes from
  FAP 2004}, volume 311 of {\em Lecture Notes on Control and Information
  Sciences}, chapter Compositional modelling of distributed-parameter systems,
  pages 115--154.
\newblock Springer, 2005.

\bibitem{MASCHKE_CPDE_2013}
B.~Maschke and A.J. van~der Schaft.
\newblock On alternative {P}oisson brackets for fluid dynamical systems and
  their extension to {S}tokes-{D}irac structures.
\newblock {\em IFAC Proceedings Volumes}, 46(26):109 -- 114, 2013.

\bibitem{Maschke_IFAC_WC20_IPHS}
B.~Maschke and A.J. van~der Schaft.
\newblock Linear {B}oundary {P}ort {H}amiltonian systems defined on
  {L}agrangian submanifolds.
\newblock {\em IFAC-PapersOnLine}, 53(2):7734--7739, 2020.
\newblock 21th IFAC World Congress.

\bibitem{Maxwell_1873_treatise}
J.C.~Maxwell.
\newblock {\em A Treatise on Electricity and Magnetism}, volume~1.
\newblock Clarendon Press, 1873.

\bibitem{Olver93}
P.J. Olver.
\newblock {\em Applications of {L}ie Groups to {D}ifferential {E}quations},
  volume 107 of {\em Graduate texts in mathematics}.
\newblock Springer, New-York, ii edition, 1993.
\newblock ISBN 0-387-94007-3.

\bibitem{ran}
A. C. M. Ran and L. Rodman.
\newblock Factorization of matrix polynomials with symmetries.
\newblock {\it IMA Preprint Series}, no. 993, 1992.


\bibitem{Rashad_IMA18_20yearsBPHS}
F.~Rashad, R.and~Califano, A.~van~der Schaft, and S.~Stramigioli.
\newblock {Twenty years of distributed port-{H}amiltonian systems: a literature
  review}.
\newblock {\em IMA Journal of Mathematical Control and Information}, 07 2020.
\newblock dnaa018.

\bibitem{Schoeberl_JMathPhysics2018}
M.~Sch\"oberl and K.~Schlacher.
\newblock On the extraction of the boundary conditions and the boundary ports
  in second-order field theories.
\newblock {\em Journal of Mathematical Physics}, 59(10):102902, 2018.

\bibitem{Jeltsema_FTSC_14}
A.~van~der Schaft and D.~Jeltsema.
\newblock Port-{H}amiltonian systems theory: An introductory overview.
\newblock {\em Foundations and Trends in Systems and Control}, 1(2-3):173--378,
  2014.

\bibitem{SCL_18}
A.~van~der Schaft and B.~Maschke.
\newblock Generalized {P}ort-{H}amiltonian {DAE} systems.
\newblock {\em Systems \& Control Letters}, 121:31--37, 2018.

\bibitem{Schaft2020_DiracLagrangeNonlinear}
A.J. ~van~der Schaft and B.~Maschke.
\newblock Dirac and {L}agrange algebraic constraints in nonlinear
  {P}ort-{H}amiltonian systems.
\newblock {\em Vietnam Journal of Mathematics}, 48(4):929--939, 2020.

\bibitem{schaft_ArXiv22_ImplPHSMonotoneDiss}
A.~van~der Schaft and V.~Mehrmann.
\newblock Linear port-{H}amiltonian {DAE} systems revisited.
\newblock {\em Systems \& Control Letters}, 177:105564, 2023.

\bibitem{Schaft_SIAM11}
A.~van~der Schaft and P.~Rapisarda.
\newblock State maps from integration by parts.
\newblock {\em SIAM Journal on Control and Optimization}, 49(6):2415--2439,
  2011.

\bibitem{schaftAEU95}
A.J. van~der Schaft and B.M. Maschke.
\newblock The {H}amiltonian formulation of energy conserving physical systems
  with external ports.
\newblock {\em Archiv f{\"u}r Elektronik und {\"U}bertragungstechnik},
  49(5/6):362--371, 1995.

\bibitem{schaftGeomPhys02}
A.J. van~der Schaft and B.M. Maschke.
\newblock Hamiltonian formulation of distributed parameter systems with
  boundary energy flow.
\newblock {\em J. of Geometry and Physics}, 42:166--174, 2002.

\bibitem{schaft_LHMNLC21_DiffDiracOp}
A.J. van~der Schaft and B.M. Maschke.
\newblock Differential operator {D}irac structures.
\newblock {\em IFAC-PapersOnLine}, 54(19):198--203, 2021.
\newblock 7th IFAC Workshop on Lagrangian and Hamiltonian Methods for Nonlinear
  Control LHMNC 2021.

\bibitem{VillegasPhD07}
J.A. Villegas.
\newblock {\em A {P}ort-{H}amiltonian {A}pproach to {D}istributed {P}arameter
  {S}ystems}.
\newblock PhD thesis, University of Twente, Enschede, The Netherlands, May
  2007.

\bibitem{Villegas_IEEE_TAC_09}
J.A. Villegas, H.~Zwart, Y.~Le~Gorrec, and B.~Maschke.
\newblock Stability and stabilization of a class of boundary control systems.
\newblock {\em IEEE Transaction on Automatic Control}, 54(1):142--147, 2009.

\bibitem{Vincent_IFAC_WC20_PhaseFields}
B.~Vincent, F.~Couenne, L.~Lef{\`e}vre, and B.~Maschke.
\newblock {Port {H}amiltonian systems with moving interface: a phase field
  approach}.
\newblock {\em IFAC-PapersOnLine}, 53(2):7569--7574, 2020.
\newblock 21th IFAC World Congress.

\bibitem{TrangVU15_2_MCMDS}
T..N.M. Vu, L. Lef{\`e}vre, and B. Maschke.
\newblock Port-{H}amiltonian formulation for systems of conservation laws:
  application to plasma dynamics in tokamak reactors.
\newblock {\em Mathematical and Computer Modelling of Dynamical Systems
  (MCMDS)}, 22 Iss. 3:181--206, 2016.
  
\bibitem{trentelman}
H.L. Trentelman, P. Rapisarda,
\newblock New algorithms for polynomial $J$-spectral factorization.
\newblock {\it Math. Control Signals Systems},
12, 24--61, 1999.

\bibitem{Weinstein_71}
A.~Weinstein.
\newblock Symplectic manifolds and their {Lagrangian} submanifolds.
\newblock {\em Advances in Mathematics}, (6):329--346, 1971.

\bibitem{Weinstein_Book_77}
A.~Weinstein.
\newblock Lectures on symplectic manifolds.
\newblock {\em CBMS Conference Series}, 29, 1977.

\bibitem{Willems_SIAM98}
J.~C. Willems and H.~L. Trentelman.
\newblock On quadratic differential forms.
\newblock {\em SIAM Journal on Control and Optimization}, 36(5):1703--1749,
  1998.

\bibitem{Zwart_ESAIM_10}
H.~Zwart, Y.~Le~Gorrec, B.~Maschke, and J.~Villegas.
\newblock {Well-posedness and regularity of hyperbolic boundary control systems
  on a one-dimensional spatial domain }.
\newblock {\em {ESAIM-Control Optimization and Calculus of Variations}},
  {16}({4}):{1077--1093}, {Oct.Dec} {2010}.

\end{thebibliography}

\end{document}